\documentclass[12pt,oneside,a4paper,reqno]{amsart}

\usepackage{fullpage}
\usepackage{hyperref}
\usepackage{amssymb,amsmath,mathrsfs,amsthm}
\usepackage[dvipsnames]{xcolor}
\usepackage{etoolbox}
\usepackage{fullpage}
\usepackage{calrsfs}
\usepackage[T1]{fontenc} 
\usepackage{amsfonts}
\usepackage{mathtools}
\usepackage{bbm}
\usepackage{bm}
\usepackage{tikz-cd}
\usepackage[shortlabels]{enumitem}
\setenumerate{label={\normalfont(\textit{\roman*})}, ref={\textrm{\roman*}}
}

\usepackage{mathtools}

\DeclarePairedDelimiter\abs{\lvert}{\rvert}%
\DeclarePairedDelimiter\norm{\lVert}{\rVert}%

\makeatletter
\let\oldabs\abs
\def\abs{\@ifstar{\oldabs}{\oldabs*}}
\let\oldnorm\norm
\def\norm{\@ifstar{\oldnorm}{\oldnorm*}}
\makeatother

\newtheorem{introtheorem}{Theorem}

\theoremstyle{definition}

\newtheorem*{introdefinition*}{Definition}
\theoremstyle{plain}

\newcommand\Z{{\mathbf Z}}

\newcommand\N{{\mathbf N}}

\newcommand\A{{\mathcal A}}
\renewcommand\L{{\mathcal L}}

\newcommand{\AP}{\mathrm{AP}}
\newcommand{\Ker}{\mathrm{Ker}}
\newcommand{\NK}{\mathrm{NKer}}

\newcommand{\prt}{\mathrm{A}}
\newcommand{\brt}{\mathrm{C}}
\newcommand{\mg}{\mathrm{S}}

\newcommand{\dsep}{\Delta}
\newcommand{\udens}{\overline{D}}
\newcommand{\Sep}{\ensuremath{\mathrm{Sep}}}
\newcommand{\oac}{\ensuremath{\overline{\mathrm{ac}}}}
\newcommand{\uac}{\ensuremath{\underline{\mathrm{ac}}}}
\newcommand{\ac}{\ensuremath{\mathrm{{ac}}}}
\newcommand{\Dim}{\ensuremath{\mathrm{dim}}}
\renewcommand{\d}{\ensuremath{\mathrm{d}}}

\newcommand{\fp}{\boldsymbol{1}}
\newcommand\C{{\mathrm C}} 

\newcommand{\bigO}{O}
\newcommand{\bigT}{\Theta}
\theoremstyle{plain}
\newtheorem{theorem}{Theorem}[section]
\newtheorem{proposition}[theorem]{Proposition}
\newtheorem{lemma}[theorem]{Lemma}

\theoremstyle{definition}
\newtheorem{definition}[theorem]{Definition}
\newtheorem{remark}[theorem]{Remark}

\newtheorem{example}[theorem]{Example}
\newtheorem*{excont}{Example \continuation}
\newcommand{\continuation}{??}
\newenvironment{continueexample}[1]
 {\renewcommand{\continuation}{\ref{#1}}\excont[continued]}
 {\endexcont}

\renewcommand{\geq}{\geqslant}
\renewcommand{\leq}{\leqslant}

\begin{document}


\title{Tameness, nullness, and amorphic complexity of automatic systems}

\author[Maik Gr\"{o}ger]{Maik Gröger}
\author[El\.{z}bieta Krawczyk]{Elżbieta Krawczyk}

\address[Maik Gr\"{o}ger]{Faculty of Mathematics and Computer Science, Jagiellonian University in Krak\'ow, Poland}
\email{maik.groeger@im.uj.edu.pl}

\address[El\.{z}bieta Krawczyk]{Faculty of Mathematics\\
University of Vienna, Austria \& 
Faculty of Mathematics and Computer Science\\Institute of Mathematics\\
Jagiellonian University in Krak\'ow, Poland}
\email{ela.krawczyk7@gmail.com}


\begin{abstract} 
In topological dynamics, tame and null systems arise naturally in the study of low-complexity aperiodic behaviour, yet providing concrete and easily testable conditions to establish their existence in a canonical class of systems is often nontrivial. 
We give a complete characterisation of tameness and nullness for minimal automatic systems generated by primitive constant length substitutions in terms of amorphic complexity---a numerical invariant recently introduced to study zero entropy systems. We derive an easily computable closed formula for this invariant in this setting and show that, for infinite automatic systems, tameness and nullness are equivalent to its value being one.
\end{abstract}

\maketitle


\section{Introduction}\label{sec:introduction}

Inspired by the seminal work \cite{BourgainFremlinTalagrand1978} of Bourgain, Fremlin, and Talagrand, K\"ohler introduced in \cite{Koehler1995} the notion of \emph{tame} dynamical systems (originally she used the term \emph{regular}).
Later, by adopting one of the main results of \cite{BourgainFremlinTalagrand1978} to the dynamical context, Glasner and Megrelishvili \cite{GlasnerMegerlishvili2006} provided the following dichotomy: the enveloping semigroup of \emph{any} system is either ``large'' (containing the Stone--\v{C}ech compactification of the natural numbers) or ``small'' (with its topology determined by sequential convergence), the latter case corresponding to the system being tame.

A particularly useful alternative characterization of tameness is provided by the theory of topological independence \cite{KerrLi2007}. Here, the existence or absence of large independence sets forces high or low dynamical complexity. For example, positive topological entropy of a system $T:X\to X$ is tied to the independence of iterates  $T^n$, $n\in\Z$ along subsets of positive density. At a weaker level, the absence of infinite independence sets is linked to tameness, while \emph{nullness} in the sense of Goodman \cite{Goodman} (i.e.\ zero topological sequence entropy) requires the lack of arbitrarily large finite independence sets, see Section \ref{sec:tamenull} or \cite{KerrLi2007} for precise definitions and further information.

Since their introduction—and in particular in recent years—tame and null systems have been widely studied. This research has focused on characterizing their structural properties in various ways and developing more directly applicable criteria for identifying tame or null systems.
For more information on tame systems, see e.g.~ \cite{Glasner2006,GlasnerMegerlishvili2006,Huang2006,Glasner2007,Glasner2017,GlasnerMegrelishvili2018,FuhrmannKwietniak2020,FuhrmannGlasnerJagerOertel2021,GlasnerMegrelishvili2022,FuhrmannKellendonkYassawi2024,FuhrmannLiu2025, Kellendonk2025}, and for null systems, see for instance \cite{HuangLiShaoYe2003,MaassShao2007,HuangYe2009,GarciaRamos2017,FuhrmannKwietniak2020,QiuZhao2020,Leonard2025,LePavlovSchlortt2025}.

As part of the recent developments, an important focus has been on establishing concrete and computationally tractable conditions to identify tame and null systems within canonical classes of dynamical systems. 
In this article, we provide a complete characterisation of tameness and nullness for the well-studied class of minimal automatic dynamical systems \cite{Fogg2002,bookQueffelec} in terms of a single numerical invariant. 
Such systems can be equivalently described as those that are conjugate to a subshift generated by a primitive, constant length substitution (see Sections \ref{sec:substitution} and \ref{sec:subsstsems}).

The numerical quantity we will utilize is the notion of amorphic complexity, a relatively new topological invariant introduced for the study of zero entropy systems \cite{FuhrmannGroegerJaeger2016}. It is tailor-made for analysing minimal systems with discrete spectrum and continuous eigenfunctions \cite{FuhrmannGroeger2020}. Beyond $\Z$-actions, this class encompasses dynamical systems canonically induced by mathematical quasicrystals, such as Penrose and chair tilings. In recent work \cite{FuhrmannGroegerJaegerKwietniak2023}, amorphic complexity was extended to these more general group actions; the authors of \cite{BaakeGaehlerGohlke2025} used this to calculate amorphic complexity for some prominent examples, including the newly discovered Hat tiling \cite{SmithMyersKaplanGoodman-Strauss2024}, see also Remarks \ref{rem:inf_tiling_and_beyond_mean_eq} and \ref{rem:Baakeresult}. Very recently, \cite{Gaehler2026} provided numerical evidence that amorphic complexity, together with another finitary invariant—asymptotic composants—is strong enough to completely distinguish minimal one-dimensional inflation tilings with pure-point spectrum for a class of small ternary Pisot unit inflation factors.
Furthermore, building on the ideas of \cite{FuhrmannGroegerJaeger2016}, the authors of \cite{KasjanKeler2025} were able to distinguish between several Toeplitz subshifts associated with $\mathcal{B}$-free systems—subshifts that were previously indistinguishable using classical topological invariants.

Before stating our main results, we need to introduce some notation: we first fix an \emph{alphabet}, i.e.\ a finite set $\mathcal{A}$. A \emph{substitution} $\varphi$ is a map that assigns to each letter $a \in \mathcal{A}$ a finite word in $\mathcal{A}^*$, the set of all finite words over $\mathcal{A}$. It is said to be of \emph{(constant) length} $k \geq 2$ if the word $\varphi(a)$ has length $k$ for all $a \in \mathcal{A}$. By concatenation, $\varphi$ naturally extends to a map on $\mathcal{A}^*$. We call $\varphi$ \emph{primitive} if there exists an $n \geq 1$ such that, for each $a \in \mathcal{A}$, the word $\varphi^n(a)$ contains every letter of $\mathcal{A}$. 

To each substitution $\varphi$, we can associate a canonical subset $X_{\varphi}\subset \A^{\Z}$ of the space $\A^{\Z}$  of all bi-infinite sequences with symbols in $\mathcal{A}$. Specifically, $X_{\varphi}$ is the collection of all $x \in \mathcal{A}^\Z$ such that every finite word appearing in $x$ also appears in $\varphi^n(a)$ for some $a \in \mathcal{A}$ and $n \geq 0$. Note that there is a natural action of the \emph{(left) shift} $T$ on $X_{\varphi}$, defined by
$T((x_n)_{n \in \Z}) = (x_{n+1})_{n \in \Z}$ for all $(x_n)_{n \in \Z} \in X_{\varphi}$.
In the following, when we say that $X_\varphi$ is tame or null, we mean that the \emph{dynamical system} $(X_\varphi, T)$ is tame or null, respectively.

For the sake of simplicity, we introduce amorphic complexity here only in the symbolic context, where it admits a direct fractal geometric interpretation. The general definition is provided in Section~\ref{sec:asymp sep numbers and ac}, along with a brief overview of some of its basic properties as a topological invariant. For a symbolic system, its \emph{upper} and \emph{lower amorphic complexity} can be defined simply as the lower and upper box dimension, respectively, of its maximal equicontinuous factor equipped with the Besicovitch (pseudo)metric (see Remark \ref{rem:different_version_Besicovitch}). While these two values need not coincide in general, we will see next that for a minimal automatic system $X_\varphi$ they do, and we denote their common value by $\ac(X_\varphi)$.

\begin{introtheorem}\label{introthm:tameness} 
    Suppose $\varphi\colon \A\to\A^*$ is a primitive substitution of length $k\geq 2$. Then, the following are equivalent:
    \begin{enumerate}
        \item $\ac(X_{\varphi})\in \{0,1\}$ (with $\ac(X_{\varphi})=0$ corresponding to $X_{\varphi}$ being  finite),
        \item $X_{\varphi}$ is null,
        \item $X_{\varphi}$ is tame.
    \end{enumerate}
\end{introtheorem}

In particular, Theorem~\ref{introthm:tameness} yields that a system $X_{\varphi}$ is null if and only if it is tame (recall that nullness always implies tameness); to the best of our knowledge, this is a new result even in case of a binary alphabet. Furthermore, at the end of the introduction of \cite{FuhrmannKellendonkYassawi2024}, it is already pointed out that \cite[Thm.~1.2]{FuhrmannKellendonkYassawi2024} together with \cite[Thm.~1.1]{FuhrmannGroeger2020} yields the equivalence of (i) and (iii) in the binary case. This means that, concerning this equivalence, Theorem~\ref{introthm:tameness} extends this observation to the full general case in the minimal setting.

Theorem~\ref{introthm:tameness} offers a natural interpretation of tameness and nullness in the context of minimal automatic systems. As shown in \cite[Cor.~1.3]{FuhrmannGroeger2020}, amorphic complexity of such systems is either $0$ or lies in the interval $[1, \infty]$. Furthermore, it is straightforward to demonstrate that amorphic complexity can attain a dense set of values in $[1, \infty)$ in the class of minimal automatic systems over a fixed alphabet of size at least two (see Proposition~\ref{prop:denseness}). Beyond the trivial case where amorphic complexity is $0$ (corresponding to periodic behaviour), this reveals a complexity gap between $0$ and $1$, with tame and null systems occupying exactly the minimal nontrivial complexity level of $1$.

A promising direction for future research would be to extend the class of systems in which the equivalence of Theorem~\ref{introthm:tameness} still holds. Even within the class of substitution systems—but beyond constant length—such an extension would be quite valuable: As already discussed at the end of \cite[Sec.~6]{FuhrmannGroeger2020}, amorphic complexity of the substitution shift $X_\tau$ for the Tribonacci substitution $\tau$ must be at least $2$, since its maximal equicontinuous factor is the two-torus. If Theorem~\ref{introthm:tameness} were to hold in this more general setting, this would imply that $X_\tau$ is neither null nor tame. This aligns well with the result of \cite[Thm.~6]{KamaeZamboni2002}, which shows that $X_\tau$ is not null, and with recent work by J.~Sudbery and R.~Yassawi demonstrating the non-tameness of $X_\tau$\footnote{private communication; see also the final report of the  \href{https://www.birs.ca/events/2025/5-day-workshops/25w5437}{BIRS workshop Directions in Aperiodic Order}}. However, the latter seems to contradict the results announced in \cite[Rem.~6.19]{GlasnerMegrelishvili2018} and \cite[p.~5]{GlasnerMegrelishvili2022}.

To establish Theorem \ref{introthm:tameness} we need to determine amorphic complexity for arbitrary minimal automatic systems.
For this, to each primitive substitution $\varphi$ of constant length we associate its \emph{discrepancy substitution} $\varphi_s$ and define the \emph{discrepancy rate} $\lambda_s$ of $\varphi$ as the dominant eigenvalue of the incidence matrix of $\varphi_s$ (see Section \ref{sec:discrepancysubstitution}). 
The discrepancy substitution $\varphi_s$ is straightforward to compute; however, in general, it is not primitive, which makes its analysis considerably more delicate.
For this reason, we introduce suitable dynamically defined pseudometrics that allow for a refined analysis of the discrepancy structure.
With this we derive a closed formula for the amorphic complexity of any minimal automatic systems, showing, in particular, that it always exists.

\begin{introtheorem}\label{introthm:ac_of_automatic_systems} Let $\varphi$ be a primitive substitution of length $k\geq 2$.  The amorphic complexity of $X_{\varphi}$ is given by
\[\ac(X_{\varphi})=\frac{\log k}{\log k - \log \lambda_s},\]
where $\lambda_s$ is the discrepancy rate of $\varphi$.\footnote{Assuming the usual conventions when dividing by 0 or $\infty$ and $-\log 0 = \infty$.}
\end{introtheorem}

We note that in case of a binary alphabet, Theorem~\ref{introthm:ac_of_automatic_systems} coincides with \cite[Thm.~1.1]{FuhrmannGroeger2020}. Moreover, for general alphabets, \cite{FuhrmannGroeger2020} also provides a potential algorithm to compute concrete lower and upper bounds for the lower and upper amorphic complexity of $X_\varphi$, whereas \cite{BaakeGaehlerGohlke2025} gives improved lower and upper bounds; see also Remark~\ref{rem:Baakeresult}.

The paper is organized as follows. In Section~\ref{sec:setup} we collect the necessary preliminaries. In Section~\ref{sec:discrepancytools} we introduce the discrepancy substitution and new dynamically generated pseudometrics, which serve as the main technical tools throughout the paper. In Section~\ref{sec:complexityofsubstitution} we prove Theorem~\ref{introthm:ac_of_automatic_systems} concerning the formula for the amorphic complexity. Finally, in Section~\ref{sec:tamenessandnullness} we establish Theorem~\ref{introthm:tameness}.

\section*{Acknowledgements}  We thank Michael Baake, Gabriel Fuhrmann, Franz G\"ahler, and Reem Yassawi for illuminating and helpful discussions. The second author was supported by National Science Center, Poland under grant no.~UMO2021/41/N/ST1/04295. 
For the purpose of Open Access, the authors have applied a CC-BY public copyright license to any Author Accepted Manuscript (AAM) version arising from this submission.

\section{The setup}\label{sec:setup}

\subsection{Asymptotic notations} 

We denote by $\Z$ the set of integers and by $\N=\{0,1,2,\ldots\}$ the set of nonnegative integers.
We use standard asymptotic notations: Let $(a_n)_{n\in\N}$ and $(b_n)_{n\in\N}$ be two sequences of nonnegative real numbers. We write
\begin{itemize}
	\item $a_n = \bigO(b_n)$ if and only if $a_n\leq Cb_n$ for some $C>0$ and all sufficiently large $n$,
	\item $a_n = \bigT(b_n)$ if and only if $a_n = \bigO(b_n)$ and $b_n = \bigO(a_n)$,
	\item for positive $(b_n)_n$, $a_n \sim b_n$ if and only if $a_n / b_n\to 1$ as $n\to\infty$.
\end{itemize}

\subsection{Topological and symbolic dynamics} Let $f:X\to X$ be a homeomorphism of a compact metric space $(X,d)$. 
We call the pair $(X,f)$ a \emph{(topological) dynamical system}.
Given two systems $(X,f)$ and $(Y,g)$, we say  $(Y,g)$ is a \emph{factor}
of $(X,f)$ if there is a continuous onto map (a \emph{factor map}) $h:X\to Y$ such that $h\circ f=g\circ h$.
If additionally, $h$ is invertible, then $h$ is a \emph{conjugacy} and we say
$(X,f)$ and $(Y,g)$ are \emph{conjugate}.

A subset $E\subseteq X$ is \emph{$f$-invariant} if it is closed and $f(E)= E$.
We say $E$ is \emph{$f$-minimal} if it is $f$-invariant and does not contain any
nonempty proper subset which is $f$-invariant.
In case that $X$ is $f$-minimal itself, we also say that $(X,f)$ is \emph{minimal}.
Moreover, if there exists $x\in X$ such that its \emph{orbit} $\{f^n(x):n\in\Z\}$
is dense in $X$, we call $(X,f)$ \emph{transitive}.

An \emph{$f$-invariant} measure for $(X,f)$ is a Borel
probability measure $\mu$ on $X$ such that $\mu(f^{-1}A)=\mu(A)$ for every Borel
set $A\subseteq X$.
The system $(X,f)$ is said to be \emph{uniquely ergodic} if it admits a unique
$f$-invariant probability measure.
Given an $f$-invariant measure $\mu$, we say that $(X,f)$ has
\emph{discrete} (or \emph{pure-point}) \emph{spectrum with respect to $\mu$} if the associated
Koopman operator $U_f$ on $L^2(X,\mu)$, defined by $U_f\varphi=\varphi\circ f$ for $\varphi\in L^2(X,\mu)$, has discrete spectrum, i.e., if $L^2(X,\mu)$ is the closed linear span of eigenfunctions of $U_f$.

Furthermore, a system $(X,f)$ is \emph{equicontinuous} if the family $(f^n)_{n\in\Z}$ is uniformly equicontinuous, that is, if for all $\nu>0$ there is $\delta>0$ such that for all $x,y\in X$ with $d(x,y)<\delta$ we have $d(f^n(x),f^n(y))<\nu$ ($n\in\Z$).
It is well known, see for instance \cite[Thm.\ 2.1]{Downarowicz2005}, that every
dynamical system $(X,f)$ has a unique (up to conjugacy) 
\emph{maximal equicontinuous factor (MEF)}: an equicontinuous factor $(Y,g)$ so
that every other equicontinuous factor of $(X,f)$ is also a factor of $(Y,g)$.

In what follows we will consider bi-infinite shift spaces $\Sigma = \mathcal{A}^\Z$, 
where $\mathcal{A}$ is a finite set referred to as the \emph{alphabet}, and 
$\Sigma$ is endowed with the product topology. This topology can be induced by 
equipping $\Sigma$ with the standard \emph{Cantor metric} defined for $x = (x_k)_{k\in\Z}$ and 
$y = (y_k)_{k\in\Z}$ in $\mathcal{A}^{\Z}$ as 
\[d_\beta(x,y) = \beta^{-j}\quad \text{where}\quad \beta > 1\quad \text{and}\quad
j = \min\{\abs{\ell} \mid x_{\ell} \neq y_{\ell}\}.\]
	The pair 
$(\Sigma, T)$ is called the \emph{full shift} over the alphabet $\mathcal{A}$ 
where $T: \Sigma \to \Sigma$ is the \emph{(left) shift} given by 
$T((x_n)_{n\in\Z}) = (x_{n+1})_{n\in\Z}$ for all $(x_n)_{n\in\Z} \in \Sigma$.
	Restricting the shift to a $T$-invariant subset $X\subseteq\Sigma$,
we will use  the same symbol $T$ for this restricted map and we will
refer to $(X,T)$ as a \emph{subshift}.
	One way to obtain subshifts is to consider \emph{orbit closures}: given $x\in\Sigma$,
we define its orbit closure to be
$\overline{\{T^k(x) : k\in\Z\}}\subseteq\Sigma$. For a subshift $(X,T)$ and a  letter $a\in \mathcal{A}$, the set $[a] = \{x\in X\mid x_0= a\}$ is referred to as the \emph{cylinder} of $a$.
In the following, if it is clear from the context, we will often omit the shift map $T$ when referring to a subshift (respectively, symbolic system) $(X,T)$.

Finally, for an alphabet $\A$, we let $\A^*$ denote the set of all finite words over $\A$.
This is a monoid under word concatenation with the neutral element given by the empty word $\epsilon$. For a word $w\in\A^{*}$, we denote by $\abs{w}$ its \emph{length}, i.e.\ the number of symbols in $w$.
For $n \in \N$, we also let $\mathcal{A}^n$ denote the set words of length $n$ over $\A$. For a nonempty word $u\in\A$, we denote by $u^{\omega}$ the biinfinite periodic sequence $u^{\omega} = \dots uuu.uuu\dots$, where the dot indicates the 0th position.

\subsection{Tameness and nullness}\label{sec:tamenull}

	Let $(X,f)$ be a dynamical system. For the definitions of tameness and nullness we follow  \cite{KerrLi2007}, referring there for a more comprehensive treatment.  A set $S\subset \Z$ is called an \emph{independence set} for the pair $(U_0,U_1)$ with $U_0,\, U_1\subset X$ if for any finite subset $G\subset S$ and any choice function $\sigma\colon G\to\{0,1\}$ one has
	\[\bigcap_{i\in G} f^{-i}U_{\sigma(i)}\neq \emptyset.\]
	A pair $(x_0,x_1)$ of points of $X$ is called an \emph{IT pair} (resp.\ \emph{IN pair})  if for any neighbourhood $U_0\times U_1$ of $(x_0,x_1)$, the pair $(U_0,U_1)$ has an infinite (resp.\ arbitrary large finite) independence set. 
A pair $(x_0,x_1)$ is said to be \emph{trivial} if $x_0=x_1$.
	
\begin{definition}\label{def:general tame null}
	A dynamical system $(X,f)$ is  \emph{tame} (resp.\ \emph{null}) if it has no nontrivial IT (resp.\ IN) pairs.\footnote{Definition \ref{def:general tame null} should be compared with the following result: a system has zero topological entropy if and only if it has no nontrivial IE pairs, where the notion of an IE pair requires the existence of an independence set of positive density, see \cite{KerrLi2007}.}
\end{definition}

	With Definition \ref{def:general tame null} it is clear that every null system is tame. For ease of exposition we make the following straightforward observation about nullness of minimal subshifts, see also \cite[Prop.\ 5.4]{KerrLi2007}.
	For a sequence $x\in \A^{\Z}$ and a finite ordered set $G=\{n_1<\dots<n_t\}\subset \Z$ we let 
\[\L(x,G)= \{T^i(x)|_{G}\mid i\in\Z\} =\{w\in\A^t\mid x_{i+n_1}\dots x_{i+n_t} = w \text{ for some } i\in\Z\} 
\] denote the set of all words that appear ``along'' $G$ in some shift of $x$.

\begin{lemma}\label{lem:nullness}
	A minimal subshift $X\subset\A^{\Z}$ is null if and only if there exists $t\in \N$ such that for some $x\in X$ (equivalently for all $x\in X)$ we have 
\begin{equation}\label{eq:nullforsubshifts}
\{a,b\}^t \nsubseteq \L(x,G) \quad \text{for all distinct } a,b\in\A \text{ and all }G\subset \Z \text{ of size } t. 
\end{equation}
\end{lemma}

\begin{definition}\label{def:t-null}	We  say that a sequence $x$ in a subshift $X\subset \A^{\Z}$ is $t$-\emph{null} if it satisfies \eqref{eq:nullforsubshifts}.
It is \emph{null} if it is $t$-null for some $t\in\N$.
\end{definition}
	Clearly, if a sequence is $t$-null, then it is $T$-null for any $T\geq t$ and a minimal subshift is null if and only if it contains a null sequence.

\subsection{Pseudometrics}
Let \(X\) be a set and let \(d\) be a pseudometric on \(X\). Then \(d\) induces an
equivalence relation on $X$ given by \(x\sim_d y\) if \(d(x,y)=0\). We denote by \([X]_d\)
the corresponding quotient space, endowed with the metric induced by \(d\);
the canonical projection \([\cdot]_d: X\to[X]_d\) is continuous and open.

Given two (pseudo)metrics \(d\) and \(d'\) on \(X\), we say that they are
(\emph{Lipschitz}) \emph{equivalent} if the identity map 
$\mathrm{id}\colon (X,d)\to (X,d')$ is a (bi-Lipschitz) homeomorphism between the topological spaces induced by \(d\) and \(d'\). 
For Lipschitz equivalence, this means that there exist constants $c, C > 0$ such that for all $x,y \in X$, we have
\[
    c \cdot d(x,y) \leq d'(x,y) \leq C \cdot d(x,y).
\]
We stress that Lipschitz equivalence is strictly stronger
than \emph{uniform equivalence} even when the underlying space \(X\) is compact;
here, \emph{uniform equivalence} means that the identity map and its inverse are
uniformly continuous.

\subsection{Box dimension}

Later, we will make use of fractal geometric ideas and arguments following \cite{FuhrmannGroeger2020}. 
To this end, we briefly introduce the concept of box dimension for a totally bounded subset $E$ of a general metric space $(M,\rho)$. 
A subset $S$ of $M$ is called \emph{$\varepsilon$-separated} if for all distinct $s,s' \in S$, we have $\rho(s,s') \geq \varepsilon$. 
We denote by $M_\varepsilon(E)$ the maximal cardinality of an $\varepsilon$-separated subset of $E$. 
The \emph{lower} and \emph{upper box dimension} of $E$ are then defined as
\begin{align}\label{eq: definition box dimension}
    \underline\Dim_B(E)=\varliminf\limits_{\varepsilon\to 0}
        \frac{\log M_\varepsilon(E)}{-\log\varepsilon}
    \qquad\textnormal{ and }\qquad
    \overline\Dim_B(E)=\varlimsup\limits_{\varepsilon\to 0}
        \frac{\log M_\varepsilon(E)}{-\log\varepsilon}.
\end{align}
If $\underline\Dim_B(E)$ and $\overline\Dim_B(E)$ coincide, their common value, 
denoted by $\Dim_B(E)$, is referred to as the \emph{box dimension}\footnote{
The standard definition of box dimension considers the smallest number of sets with 
diameter strictly smaller than $ \varepsilon $ required to cover $ E $. The above definitions 
remain unchanged if $M_\varepsilon(E)$ in \eqref{eq: definition box dimension}
is replaced by this number (see, for example, Proposition~1.4.6 in \cite{Edgar1998}).} of $E$.

It is well known that both the lower and upper box dimension are invariant under
bi-Lipschitz homeomorphisms. 
In particular, this applies to the case of two Lipschitz equivalent metrics on a compact metric space.

\subsection{Amorphic complexity}\label{sec:asymp sep numbers and ac}

We briefly introduce amorphic complexity for $\Z$-actions and refer the reader to \cite{FuhrmannGroegerJaeger2016} 
and \cite{FuhrmannGroegerJaegerKwietniak2023} for a detailed discussion and various examples.
Note that \cite{FuhrmannGroegerJaeger2016} deals with $\N$-actions, while \cite{FuhrmannGroegerJaegerKwietniak2023}
covers actions by general locally compact $\sigma$-compact amenable groups.

Given a dynamical system $(X,f)$, $x,y\in X$ and $\delta>0$, we set
\begin{equation*}
	\dsep(f,\delta,x,y)=\left\{\ell\in\Z \mid d(f^{\ell}(x),f^{\ell}(y))\geq\delta\right\}.
\end{equation*}
For $\nu\in(0,1]$ we say $x$ and $y$ are \emph{$(f,\delta,\nu)$-separated}
if $\udens(\dsep(f,\delta,x,y))\geq\nu$, where $\udens(E)$ denotes the
\emph{upper density} of a subset $E\subseteq\Z$ defined as
\begin{equation*}
	\udens(E)=\varlimsup\limits_{n\to\infty}\frac{\abs{E\cap[-n,n]}}{2n+1}.
\end{equation*} 
A subset $S\subseteq X$ is said to be \emph{$(f,\delta,\nu)$-separated} if all pairs of
distinct $x,y\in S$ are $(f,\delta,\nu)$-separated.  
The \emph{(asymptotic) separation numbers} of $(X,f)$, denoted by $\Sep(f,\delta,\nu)$
for $\delta>0$ and $\nu\in(0,1]$, are defined as the largest cardinality of
an $(f,\delta,\nu)$-separated set contained in $X$.
If $\Sep(f,\delta,\nu)$ is finite for all $\delta>0,\nu\in(0,1]$, we say $(X,f)$ has
{\em finite separation numbers}, otherwise we say it has {\em infinite separation numbers}.
Systems with positive entropy or with a nontrivial weakly mixing measure
have infinite separation numbers, see \cite{FuhrmannGroegerJaeger2016,
FuhrmannGroegerJaegerKwietniak2023} for further information.
 
For systems with finite separation numbers, we 
define the \emph{lower} and \emph{upper amorphic complexity} of $(X,f)$ as
\begin{equation*}
	\uac(f)=\adjustlimits\sup_{\delta>0}\varliminf_{\nu\to 0}
		\frac{\log \Sep(f,\delta,\nu)}{-\log \nu}
	\quad\textnormal{and}\quad
	\oac(f)=\adjustlimits\sup_{\delta>0}\varlimsup_{\nu\to 0}
		\frac{\log\Sep(f,\delta,\nu)}{-\log \nu}. 
\end{equation*}
If both values coincide, we call $\ac(f)=\uac(f)=\oac(f)$ the {\em amorphic
complexity} of $(X,f)$. 
By allowing the above quantities to assume values
in $[0,\infty]$, they are  well defined for any map $f:X\to X$.
In particular, systems with infinite separation numbers have infinite amorphic
complexity.

The proof of the following statements is verbatim as the proofs for the
corresponding statements for $\N$-actions, see 
\cite[Prop.~3.4 \& Prop.~3.7]{FuhrmannGroegerJaeger2016}.

\begin{proposition}\label{prop: properties ac}
	Let $(X,f)$ and $(Y,g)$ be dynamical systems.
	\begin{enumerate}
			\item\label{prop: properties ac1} Suppose $(Y,g)$ is a factor of $(X,f)$.
					Then $\oac(f)\geq\oac(g)$ and $\uac(f)\geq \uac(g)$.
					In particular, amorphic complexity is an invariant of topological
					conjugacy.
			\item\label{prop: properties ac2} Let $m\in\Z$. 
		Then $\oac(f^m)=\oac(f)$ and
					$\uac(f^m)=\uac(f)$. 
	\end{enumerate}
\end{proposition}

In the minimal setting, there is a natural class of systems for 
studying amorphic complexity (meaning that they always have finite separation numbers). 
This class, in particular, includes all substitution systems with discrete spectrum (see the next two sections).  
We call a dynamical system $(X,f)$ \emph{(Besicovitch-) mean equicontinuous} if, for every $\nu > 0$, 
there exists $\delta > 0$ such that for all $x, y \in X$ with $d(x,y) < \delta$, we have 
\begin{align}\label{def:Besicovitch_pseudometric}
    D_B(x,y) =
    \varlimsup\limits_{n\to\infty}\frac{1}{2n+1}\sum\limits_{i=-n}^{n}d(f^i(x),f^i(y))<\nu.
\end{align}

For the next theorem, see the discussion in \cite[Sec.~4]{FuhrmannGroeger2020} and \cite[Sec.~3]{FuhrmannGroegerJaegerKwietniak2023}.  

\begin{theorem}\label{thm:equivalence_mean_equi_fin_sep_num}  
   Assume $(X,f)$ is minimal.  
    Then $(X,f)$ is mean equicontinuous if and only if it has finite separation numbers.  
\end{theorem}  

\begin{remark}\label{rem:folner}
    Let us note that if a dynamical system is mean equicontinuous, then the $\limsup$ in \ref{def:Besicovitch_pseudometric} is in fact a limit. Moreover, in this case, the limiting behaviour of \ref{def:Besicovitch_pseudometric} is independent of the specific choice of the Følner sequence $(\{-n,\ldots,n\})_{n\in\N}$ (see, for example, the brief discussion following \cite[Thm.~1.2]{FuhrmannGroegerLenz2022}).

    Additionally, for a mean equicontinuous system $(X,f)$, the definition of amorphic complexity (via the definition of the upper density) is independent of the particular Følner sequence used, see \cite[Thm.~4.1]{FuhrmannGroegerJaegerKwietniak2023}. In particular, one could also use the standard one-sided Følner sequence $(\{0,\ldots,n\})_{n\in\N}$, which is equivalent to considering the amorphic complexity of $(X,f)$ as an $\N$-action, see \cite{FuhrmannGroegerJaeger2016}.
\end{remark}

Finally, for the full shift $(\Sigma,T)$, where $\Sigma$ is equipped with the 
Cantor metric $d_\beta$ with $\beta > 1$, it is not difficult to show that for 
each $\delta > 0$  
\begin{align}\label{eq:besicovitch metric and density of discrepancy}
    D_{\delta}(x,y) = \udens(\dsep(T,\delta,x,y)) \quad \textnormal{for } x,y\in\Sigma
\end{align}
defines a pseudometric and that $(D_\delta)_{\delta\in(0,1]}$ forms a family of 
Lipschitz-equivalent pseudometrics, see Lemma 6.2 and its proof in 
\cite{FuhrmannGroeger2020}. Following standard procedure, we introduce an 
equivalence relation on $\Sigma$ by identifying $x,y \in\Sigma$ whenever 
$D_{\delta}(x,y) = 0$ and denote by $[\cdot]$ the corresponding quotient 
mapping. Since this relation is independent of the choice of $\delta \in (0,1]$, 
the mapping is well defined; we refer to $[X]$ as \emph{the Besicovitch space} of $X$.

\begin{theorem}[{\cite[Thm.~6.5]{FuhrmannGroeger2020}}]\label{thm: amorphic complexity and box dimension}
    Suppose $(X,T)$ is a subshift with finite separation numbers. Then the 
    (lower and upper) box dimension of the subset $[X]$ of $[\Sigma]$ 
    (equipped with $D_\delta$) is independent of $\delta\in(0,1]$. 
    Furthermore,
    \begin{align*}
        \uac(T|_X) = \underline\dim_B\left(\left[X\right]\right)
        \quad \textnormal{and} \quad
        \oac(T|_X) = \overline\dim_B\left(\left[X\right]\right).
    \end{align*}
\end{theorem}

Note that the last assertion applies particularly to minimal mean equicontinuous 
subshifts by Theorem \ref{thm:equivalence_mean_equi_fin_sep_num}. 
In fact, it is straightforward to show that a subshift $(X,T)$ is mean 
equicontinuous iff the restriction $\left.[\cdot]\right|_X$ is continuous, 
see \cite[Prop.~6.8]{FuhrmannGroeger2020}.

\begin{remark}
    From now on, for a symbolic system $(X,T)$, we will—when there is no danger of confusion—use the notation $\uac(X) = \uac(T|_X)$ and $\oac(X) = \oac(T|_X)$, as well as $\ac(X)$ if both values coincide.
\end{remark}

\begin{remark}\label{rem:inf_tiling_and_beyond_mean_eq}
    In the context of more general group actions, the authors of \cite{BaakeGaehlerGohlke2025} study amorphic complexity—for which they suggest to use the name \emph{orbit separation dimension}—for actions induced by primitive self-similar inflation tilings. 
    They determine it for several notable examples, including the recently discovered Hat tiling \cite{SmithMyersKaplanGoodman-Strauss2024}.

    To extend the analysis of asymptotic separation numbers beyond mean equicontinuous systems, Kasjan and Keller propose in \cite{KasjanKeler2025} to study (Besicovitch) covering numbers of symbolic orbits of generic points. 
    Their approach applies to the wider class of symbolic systems with discrete spectrum and reduces to the study of asymptotic separation numbers in the mean equicontinuous case. 
    It allows them to distinguish between several Toeplitz subshifts associated with $\mathcal{B}$-free systems, which were previously indistinguishable using classical topological invariants.
\end{remark}

\subsection{Substitutions}\label{sec:substitution}

	 Let $\A$ be a finite alphabet. 
	 A \emph{substitution} $\varphi\colon\mathcal{A}\to\mathcal{A}^{*}$ is a function that assigns to each letter $a\in\A$ some (possibly empty) word over $\A$.
	A letter $a\in\A$ is \emph{erasing} w.r.t.\ $\varphi$ if $\varphi^n(a)=\epsilon$ for some $n\geq 1$; otherwise it is \textit{nonerasing}.
	   A \textit{coding} is an arbitrary map $\pi\colon \mathcal{A}\rightarrow \mathcal{B}$ between alphabets $\mathcal{A}$ and $\mathcal{B}$.
	  A substitution $\varphi\colon \mathcal{A}\rightarrow \mathcal{A}^*$ induces natural maps 
\[\varphi\colon \mathcal{A}^{*} \to \mathcal{A}^{*}\quad \text{and} \quad 
\varphi\colon \mathcal{A}^{\Z} \to \mathcal{A}^{\Z}\]
by concatenation (with the convention of using the same symbol for these extensions); in the latter case it is given for $x=(x_n)_n\in\A^{\Z}$ by
	\[\varphi(x) = \dots\varphi(x_{-1}).\varphi(x_0)\varphi(x_1)\dots,
	\]
	where the dot denotes the 0th position.
	In the same way we extend a coding between two alphabets to a map between words or sequences.
	A sequence $x\in\A^{\Z}$ is called a \emph{fixed point of} $\varphi$ if $\varphi(x)=x$.

	   Let $M$ be a square matrix with nonnegative entries. 
	   Recall that $M$ is called \emph{primitive} if $M^n$ has strictly positive entries for some $n\geq 1$.
	    We say that $M$ is in \emph{normal primitive (Frobenius) form} if it is conjugate via some permutation matrix to a lower block diagonal matrix with  square diagonal blocks $B_1,\dots ,B_t$ (\emph{primitive components}), which are either primitive or a $1\times 1$ zero matrix. 
	 Any matrix $M$ has some power $M^m$, $m\geq 1$ which is in normal primitive form, see e.g.\ \cite[Sec.\ 4.4 \& 4.5]{bookMarcusLind} for an excellent exposition. 
	For a nonnegative matrix $M$, we let $||M||_1$ denote the maximum column sum of $M$.

	   The \emph{incidence matrix} of $\varphi$ is given by $M_{\varphi}= (|\varphi(b)|_a)_{a,b\in \mathcal{A}}$, where $|\varphi(b)|_a$ denotes the number of occurrences of the letter $a$ in $\varphi(b)$. 
	   Note that $M_{\varphi^{n}}=M_{\varphi}^n$ for any $n\in\N$.
	 A substitution is called \emph{primitive} (resp.\  \emph{in normal primitive form}) if its incidence matrix is primitive (in normal primitive form).
	     Since $M_{\varphi}$ is nonnegative, it has the dominant (Perron--Frobenius) eigenvalue $\lambda$, i.e.\ a real nonnegative algebraic integer $\lambda$ such that $\abs{\lambda'}\leq \lambda$ for any other eigenvalue $\lambda'$ of $M_{\varphi}$. 
	    By abuse of terminology, we will call $\lambda$ also the dominant eigenvalue of $\varphi$. 
	    In the case when $\varphi$ is primitive, $\lambda$ describes the growth rate of each letter in $\mathcal{A}$ under $\varphi$. 
	    In general, the growth rates of letters can be described in terms of the dominant eigenvalues of the primitive components of  $M_{\varphi}$.

\begin{lemma}\cite[Lem.\ 9]{DurandRigo2009}\label{lem:growth_of_letters}\mbox{}
\begin{enumerate}
\item\label{lem:growth_of_letters1} For a nonnegative matrix $A$, 
\begin{equation}\label{eq:Mfor}
||A^n||_1 = \bigT (\lambda^n n^d),
\end{equation} 
where $\lambda$ is the dominant eigenvalue of $A$ and $d\in\N$.
\item Let $\varphi\colon\mathcal{A}\to\mathcal{A}^{*}$ be a substitution, and let $m\geq 1$ be such that $\varphi^m$  is in normal primitive  form. 
	For each nonerasing letter $a\in\mathcal{A}$ there exist unique algebraic integer $\lambda(a)\geq 0$,  integer $d(a)\geq 0$, and $c(a)>0$ such that
\begin{equation}\label{eq:growthtype}
\abs{\varphi^{mn}(a)} \sim c(a)\cdot(mn)^{d(a)}\cdot\lambda(a)^{mn}.
\end{equation}
	If $a\in\A$ is erasing we put $\lambda(a)=0$, $d(a)=1$, $c(a)=1$; in this case $\abs{\varphi^n(a)}=0$ for all $n$ large enough.
\end{enumerate}
\end{lemma}

\begin{definition}\label{def:typegrowth}
	For each $a\in\A$, we call the  pair $(\lambda(a),d(a))$ from Lemma \ref{lem:growth_of_letters} the \textit{growth type} of $a$ and the number $\lambda(a)$ the \textit{growth rate} of $a$.
	The \emph{growth type} (resp.\ \textit{growth rate}) of $\varphi$ is the maximal growth type (resp.\ maximal growth rate) of its letters.\footnote{Of course, the pair $(\lambda,\d)$ is the growth type of $\varphi$ if and only if  it satisfies \eqref{eq:Mfor} with $A=M_{\varphi}$.} 
	A letter $a\in A$ is of \emph{maximal growth type} (resp.\ \emph{maximal growth rate}) if its growth type (resp.\ growth rate) is the same as the growth type (resp.\ growth rate) of $\varphi$. 
	We will call letters of maximal growth rate simply of \emph{maximal growth}.
\end{definition}

\subsection*{Automatic sequences} Let $k\geq 2$. A substitution $\varphi\colon\A\to\A^{*}$ is \textit{of (constant) length $k$} if $|\varphi(a)|=k$ for each $a\in \mathcal{A}$.
	 It will be useful to introduce the following terminology originating from computer science: a sequence is called $k$-\emph{automatic} if it is a coding of some fixed point of a substitution of constant length $k$; see \cite{bookAlloucheShallit} for a general reference.

\begin{definition}\label{def:kernelmaps1}
For a sequence $x$ and $t\geq 1$, we let 
\[\AP_t(x)=\{(x_{i+tn})_n \mid 0\leq i<t\}
\] denote the set of sequences that appear in $x$ along the arithmetic progressions of difference $t$. 
	For $k\geq 2$, we let
	\[\mathrm{Ker}_k(x) = \bigcup \{\AP_{k^{m}}(x)\mid m\in\N\} = \{(x_{i+k^mn})_n \mid m\in\N,\ 0\leq i<k^m  \}
	\]
	be  the $k$-\emph{kernel} of $x$.
\end{definition}

	A classical result of Cobham says that a sequence is $k$-automatic if and only if its $k$-kernel is finite (see \cite{Cobham1972} or \cite[Thm.\ 6.6.2]{bookAlloucheShallit}).
Thus, it is clear that for a $k$-automatic sequence $x$, any element of $\Ker_{k}(x)$ is also $k$-automatic.

\begin{definition}\label{def:kernelmaps2}
 For a substitution $\varphi \colon \mathcal{A}\rightarrow \mathcal{A}^*$ of length $k\geq 2$ and each $i=0,\ldots,k-1$ we define the map
\[\varphi_i\colon \mathcal{A} \to  \mathcal{A},\quad a\mapsto \varphi(a)_i,\]
that sends a letter $a$ to the $i$th letter of $\varphi(a)$. 
	We extend $\varphi_i$ by concatenation to a map on $\mathcal{A}^{\Z}$.
	For each $m\in \N$ we define 
	\[\C_m(\varphi) = \{\varphi^m_j(\A)\mid 0\leq j<k^m\}\quad\text{and}\quad \C(\varphi)=\{\varphi^m_j(\A)\mid 0\leq j<k^m,\ m\in\N\};
	\]
we refer to the elements of $\C_m(\varphi)$ as  \emph{column sets} of $\varphi^m$.
	We say that $\varphi$ has a \emph{coincidence} if $\C(\varphi)$ contains a one-element set $\{a\}$ for some $a\in\A$.
	\end{definition}
	Clearly, the set $\C(\varphi)$ is finite.
	If $\varphi$ is primitive and $\C(\varphi)$ contains a one-element set $\{a\}$ for some $a\in\A$, then it contains all the sets $\{a\}$, $a\in \A$.
	Visually, the set $\C_m(\varphi)$ corresponds to the union of the sets of columns that appear in the $|\mathcal{A}|\times k^m$ arrays formed by the words $\varphi^m(a)$, $a\in\mathcal{A}$.
	 The substitution $\varphi$ has a coincidence if for some $m\in\N$, we see in the array a column consisting only of one letter.
	 If $x$ is a fixed point of $\varphi$, then $\Ker_k(x) =\{\varphi^m_{i}(x)\mid m\in \N,\ 0\leq i<k^{m}\}$.
	In fact, for any  $y\in \Ker_{k}(x)$ and $m\geq 1$ we have
	\begin{equation}\label{eq:APofkernelelements}\AP_{k^{m}}(y)=\{\varphi^m_{i}(y)\mid 0\leq i<k^{m}
	\}.
	\end{equation}

\subsection{Systems defined by substitutions}\label{sec:subsstsems}

	Each substitution $\varphi$ gives rise to a subshift
\[X_{\varphi}=\{x\in \mathcal{A}^{\Z}\mid \text{ each word appearing in } x \text{ appears in } \varphi^n(a) \text{ for some } a\in\mathcal{A}, n\geq 0\}.\] 
	If $\varphi$ is primitive, the subshift $X_{\varphi}$ is minimal. 
	If $\varphi$ is of constant length $k\geq 2$, then $X_{\varphi}$ is minimal if and only if $\varphi$ is primitive \cite[Sec.\ 5.2]{bookQueffelec}.
	Furthermore, for a primitive $\varphi$, $X_{\varphi}=X_{\varphi^n}$ for any $n\geq 1$ and by taking a power of $\varphi$ if necessary, we may always assume that $\varphi$ has a fixed point $x\in X_{\varphi}$.

    One can also consider a somewhat more general setup when the
 system is given by a substitution $\varphi\colon \A\to\A^*$ together with a coding $\tau\colon\A\to\mathcal{B}$.
 The subshift $X$ generated by $(\varphi,\tau)$ is given simply 
 by $X=\tau(X_{\varphi})$\footnote{From a dynamical point of view considering all systems $X_{\varphi}$ together with codings amounts to considering all systems $X_{\varphi}$ together with their topological factors.}; if one can take $\varphi$ to be of constant length $k$, then we call $X$ a $k$-\emph{automatic} system. 
    This terminology is justified by the fact that if $X$ is 
 transitive, then $X$ is $k$-automatic if and only if there exists a $k$-automatic sequence $x$ such that $X$ is the orbit closure of $x$ \cite[Lem.\ 2.10]{ByszewskiKoniecznyKrawczyk}.
    It is well known that there exists a minimal automatic 
 sequence $x$\footnote{i.e.\ $x$ such that its orbit closure is a minimal system} which is not a fixed point of any substitution: the prototypical example is the Rudin--Schapiro sequence \cite[Ex.\ 26]{AlloucheTaxonomy}. 
    Yet, due to the  result of M\"ullner and Yassawi \cite{MullnerYassawi2021} this cannot happen on the level of (minimal) dynamical systems; see Theorem \ref{thm:reemclemens} below. 
    For this reason we will restrict ourselves in our treatment only to systems of the form $X_{\varphi}$ although some of the definitions could be made more general. 

    \begin{theorem}\cite[Thm.\ 22]{MullnerYassawi2021}\label{thm:reemclemens}
     Any infinite factor of a minimal $k$-automatic system is conjugate to a system $X_{\varphi}$ given by a primitive substitution of constant length $k^m$ for some $m\geq 1$.
    \end{theorem}

    It is well known that any system $X_{\varphi}$ given by a primitive substitution $\varphi$ is uniquely ergodic; we will say that $X_{\varphi}$ has \emph{discrete spectrum} if it has discrete spectrum with respect to this unique invariant measure.
	The spectral theory of systems generated by primitive constant length substitutions is well understood due to the work of Dekking. 
	We briefly summarise the main facts here and refer to \cite{Dekking1978} for more details.

\begin{definition}\label{def:height}
Let $\varphi\colon\A\to\A^*$ be a primitive substitution of length $k\geq 2$  with a fixed point $x$.
The \emph{height} of $\varphi$ is defined by 
	\[
		h(\varphi)=h=\max\{n\geq 1 \mid \gcd(n,k)=1\ \text{and}\ n \text{ divides } \gcd\{m\mid x_0=x_m\}\}.
	\]
	For a substitution $\varphi$ of height $h$, its \emph{pure base} is the natural substitution $\varphi'\colon\A'\to(\A')^*$ of constant length $k$  defined on the alphabet $\mathcal{A}' = \{x_{nh}x_{nh+1}\dots x_{nh+h-1}\mid n\in\Z\}$ consisting of all words of 
	length $h$ that appear at positions divisible by $h$ in $x$.
	\end{definition}

    \begin{remark}\label{rem:height}
        The height $h$ of $\varphi$ can alternatively be defined as the maximal $n\geq 1$ coprime with $k$ such that $\Z/n\Z $ is a factor of $X_{\varphi}$ (with $X_{\varphi}$ not necessarily infinite); see \cite[Lem.\ 10]{Dekking1978}.
    \end{remark}
    
	The height $h$ is bounded by both the length $k$ and the size of the alphabet $A$, and it is independent of the choice of a fixed point of $\varphi$.
	The pure base $\varphi'$ always has height 1.
		
	\begin{example}[\cite{Dekking1978}]\label{ex:heightandpurebase}
		Consider the substitution
	\[\varphi(0)=010,\quad \varphi(1)=102,\quad \varphi(2)=201\]
	of constant length $3$ with its fixed point
\[ x = \dots102010201.010102010102010201\dots
\]	
The height of $\varphi$ is $2$, $\A'=\{01, 02\}$, and the pure base of $\varphi$ is given by
	\[\varphi'(01)=01|01|02,\quad \varphi'(02)=01|02|01.\]
	\end{example}

	We let $\mathbf{Z}_k$ be the ring of the $k$-adic integers, that is,  the inverse limit 
$\mathbf{Z}_k=\varprojlim\mathbf{Z}/k^n\mathbf{Z}$
 of the inverse system of rings $(\Z/k^i\Z )$, $i\geq 1$, where the morphisms  are given by the natural quotient maps.
	The ring $\Z_k$ is a compact topological ring; we may identify $\Z_k$ with the  closed subring
\[\Z_k=\{(s_i)_i\in \prod \Z/k^i\Z\mid s_i\equiv s_{i+1} \bmod k^i\}
\]
of the direct product of $\Z/k^i\Z$, $i\geq 1$. 
	Considering $\Z_k$ with the addition of $1$ makes it into an equicontinuous topological system.
Recall also that for $h\geq 1$ the \emph{$h$-height suspension} of a system $(X,T)$ is given by $(Y, S)$, where $Y = X\times \Z/h\Z$ and 
\[S(x,i)=\begin{cases} (x,i+1)\quad &\text{for}\quad i=0,\dots,h-1\\
	 (Tx,0) \quad &\text{for}\quad i=h
	\end{cases}.\]

\begin{theorem}\label{thm:dekking}
Assume $\varphi$ is a primitive substitution of  length $k\geq 2$. 
Let $h$ be its height and $\varphi'$ its pure base.
\begin{enumerate}
    \item\label{thm:dekking1} The system $X_{\varphi}$ is conjugate to the $h$-height suspension of $X_{\varphi'}$.
    \item\label{thm:dekking2} If $X_{\varphi}$ is infinite, the MEF of $X_{\varphi}$ is given by $\Z_k\times \Z/h\Z$.
    \item\label{thm:dekking3} The system $X_{\varphi}$ has discrete spectrum iff the pure base $\varphi'$ has a coincidence.
    \item\label{thm:dekking4} We have that $\uac(X_\varphi)=\uac(X_{\varphi'})$ and $\oac(X_\varphi)=\oac(X_{\varphi'})$.
    \item\label{thm:dekking5} We have that $X_{\varphi}$ is null (resp.\ tame) if and only if $X_{\varphi'}$ is null (resp.\ tame).
\end{enumerate}
\end{theorem}
\begin{proof}
    Assertions \eqref{thm:dekking1}--\eqref{thm:dekking3} and their proofs are given in \cite{Dekking1978}. 

    By \eqref{thm:dekking1}, to prove \eqref{thm:dekking4} and \eqref{thm:dekking5} it is enough to show that for any system $(X,T)$ and $h\geq 1$, $(X,T)$ is null (resp.\ tame) if and only if its $h$-height suspension $(X\times \Z/ h\Z, S)$ is  null (resp.\ tame) and that the value of amorphic complexity does not change when taking finite suspensions.
    To see this, first note that, by Proposition \ref{prop: properties ac}, for any system $(Y,S)$ the amorphic complexities of $(Y,S)$ and $(Y,S^h)$ coincide. Likewise, an easy pigeonhole argument shows that $(Y,S)$ is null (resp.\ tame) if and only if $(Y,S^h)$ is null (resp.\ tame); see also \cite{Leonard2025} for the proof for the null case.
    The claims  follow thus from the fact that the system  $(X\times \Z/h\Z, S^h)$ is conjugate to $(X\times \Z/ h\Z, T\times \mathrm{id})$, i.e.\ to the system consisting of $h$ disjoint copies of $X$.
    (The justification for \eqref{thm:dekking4} can also be found in the proof of \cite[Thm.~7.13]{FuhrmannGroeger2020} on p.~1402.) 
\end{proof}

\section{Induced pseudometrics and discrepancy substitutions}\label{sec:discrepancytools}

\subsection{Pseudometrics related to clopen partitions}\label{sec:pseudometrics}

	For a subshift $(X,T)$, let $\mathcal P$ be a clopen partition of $X$, 
meaning that $\mathcal P$ consists of finitely many nonempty clopen
subsets $P\subseteq X$ such that 
\[
	X=\bigsqcup\limits_{P\in\mathcal P} P.
\]
     
    Let $\prt\subseteq \mathcal P\ast \mathcal P$\footnote{Here, $\mathcal P\ast \mathcal P$ denotes the \emph{symmetric product} of $\mathcal P$ with itself, i.e.\
$\mathcal P\ast \mathcal P = \{ \{U,V\}\mid U,V \in \mathcal{P}\}$.} be any set of unordered pairs of partition elements and assume that $\mathcal P\subseteq \prt$.
       Consider the relation $\sim_{\prt}$ on $\mathcal P$ which identifies $U,V\in\mathcal P$ if and only if $\{U,V\}\in \prt$.
 Note that the assumptions on $\prt$ imply that $\sim_{\prt}$ is reflexive and symmetric. 
    We say that $\prt$ is \emph{transitive} if the relation $\sim_{\prt}$ is transitive; in this case $\sim_{\prt}$ is an equivalence relation
    and we let  $\mathcal B=\mathcal P/\sim_{\prt}$ denote the set of equivalence classes of $\sim_{\prt}$. 

For $x\in X$ we denote by $\mathcal{P}(x)$ the unique element $P$ of $\mathcal{P}$ such that $x\in P$.
We have that the map
\begin{equation}\label{eq:factormapfrompartition}
    \pi_{\prt}\colon X\to \mathcal{B}^{\Z},\quad (\pi_{\prt}(x))_i = [\mathcal{P}(T^i x)]_{\sim{\prt}}, \ i\in\Z
\end{equation}
is a well-defined factor map from $X$ onto the subshift $\pi_{\prt}(X)\subseteq \mathcal{B}^{\Z}$.
 	
We will now describe a somewhat dual construction to the above.
For any set  $\brt\subseteq(\mathcal P\ast \mathcal P)\backslash\mathcal P$ of unordered pairs of distinct partition elements we define the function $D_{\brt}\colon X\times X\to [0,1]$ by
\begin{equation}\label{eq:def_of_pseudo}
    D_{\brt}(x,y)=\limsup\limits_{n\to\infty}\frac{1}{2n+1}
        \sum\limits_{k=-n}^{n}\sum\limits_{\alpha\in\brt}
        \mathbb{S}_\alpha(T^k(x),T^k(y)),
\end{equation}
where for $\alpha=\{U,V\}\in\brt$ we set
\[
    \mathbb{S}_\alpha(x,y) = \mathbbm{1}_U (x)\cdot\mathbbm{1}_V (y)
    +\mathbbm{1}_V (x)\cdot\mathbbm{1}_U(y).
\]
	
Note that $D_{\brt}$ is always reflexive, symmetric, and invariant under $T$ (meaning that $D_{\brt}\circ T=D_{\brt}$).
However, $D_{\brt}$ does not always have to be a pseudometric on $X$.
To ensure this property, we say that 
$\brt\subseteq(\mathcal P\ast \mathcal P)\backslash\mathcal P$
is \emph{transitive} if for all $\{U,V\}\in\brt$ and all $W\in\mathcal P$
we either have $\{U,W\}\in\brt$ or $\{V,W\}\in\brt$, see Lemma \ref{lem:pseudometricversusfactormap}. 
Once $D_{\brt}$ is a pseudometric, we denote by $[X]_{\brt}$ the corresponding quotient space.
Since $D_{\brt}$ is invariant under $T$, we can consider the shift as a mapping on $[X]_{\brt}$, by setting $T([x]_{\brt})=[T(x)]_{\brt}$. 
This implies, in particular, that the shift becomes an isometry on $[X]_{\brt}$.

\begin{remark}\label{rem:different_version_Besicovitch}
    For a fixed partition $\mathcal{P}$ we denote by $\fp$ the full set  
    $\fp = \mathcal{P} \ast \mathcal{P} \setminus \mathcal{P}$  
    of distinct unordered pairs of elements of $\mathcal{P}$. With this convention,  
    the pseudometric $D_{\fp}$ defined in \eqref{eq:def_of_pseudo} corresponds to  
    the standard \emph{Besicovitch pseudometric}\footnote{Often, this term also  
    refers to the pseudometric defined in \eqref{def:Besicovitch_pseudometric}.  
    In the symbolic context, $D_{\fp}$ and $D_B$ are uniformly equivalent  
    (see, for instance, the proof of Proposition 6.8 in \cite{FuhrmannGroeger2020})  
    but, in general, they do not need to be Lipschitz equivalent.} on a subshift.  
    Clearly, $D_{\fp}$ coincides with $D_\delta$ from  
    \eqref{eq:besicovitch metric and density of discrepancy} for $\delta=1$ and
    for any $\brt \subseteq (\mathcal{P} \ast \mathcal{P}) \backslash \mathcal{P}$,  
    we have  
    \[
        D_{\brt}(x,y) \leq D_{\fp}(x,y),
    \]
    for all $x,y \in X$.
\end{remark}

\begin{lemma}\label{lem:pseudometricversusfactormap}
    Let $(X,T)$ be a subshift and $\mathcal P$ a clopen partition of $X$. 
    Let  $\prt\subseteq \mathcal P \ast \mathcal P$ be a set containing $\mathcal P$.
    The following conditions are equivalent:
    \begin{enumerate}
        \item\label{lem:pseudometricversusfactormap1} $\prt$ is transitive,
        \item\label{lem:pseudometricversusfactormap2} its complement $\brt = \prt'\subseteq (\mathcal P \ast \mathcal P)\setminus \mathcal P$ is transitive.
    \end{enumerate}
    Under the assumption of transitivity, $D_{\prt'}$ is a pseudometric.
    Furthermore, if $X$ is mean equicontinuous, then the quotient system $([X]_{\prt'},T)$ is conjugate to the MEF of $\pi_{\prt}(X)$.
\end{lemma}
\begin{proof}
    Clearly, for any $U,V$ in $\mathcal P$, $U\sim_{\prt} V$ if and only if $\{U,V\}\notin \prt'$.
    The negation of the statement
    \begin{center}
        $\forall U,V,W$ in $\mathcal P$ if $U\sim_{\prt} W$ and $W\sim_{\prt} V$, then $U\sim_{\prt} V$,
    \end{center}         
    is exactly  the statement
    \begin{center}
        $\exists U,V,W$ in $\mathcal P$ s.t.\ $\{U,V\}\in \prt'$ but both $\{U,W\}$ and $\{V,W\}$ do not lie in $\prt'$.
    \end{center} 
    Hence, \eqref{lem:pseudometricversusfactormap1} and \eqref{lem:pseudometricversusfactormap2} are equivalent.
	
    Assuming that $\prt'$ is transitive, we immediately get for all $x,y,z\in X$ that
\[
    \mathbb{S}_{\{U,V\}}(x,y)
        \leq\max\{\mathbb{S}_{\{U,W\}}(x,z)+\mathbb{S}_{\{V,W\}}(z,y),\,
            \mathbb{S}_{\{V,W\}}(x,z)+\mathbb{S}_{\{U,W\}}(z,y)\},
\]
for all $\{U,V\}\in\prt$, by choosing the unique $W\in\mathcal P$ that
contains $z$.
    This proves that the triangle inequality holds for $D_{\prt'}$ and, hence, $D_{\prt'}$ is a pseudometric.

    Now, assume that $X$ is mean equicontinuous, which implies that $\pi_{\prt}(X)$ is mean equicontinuous as well. Note that on $\pi_{\prt}(X)$ we have a natural partition induced by $\mathcal P$, namely, the partition given by the cylinders of letters $b\in \mathcal B=\mathcal P/\sim_{\prt}$.
    To prove the last assertion it is enough to show that the Besicovitch space $[\pi_{\prt}(X)]$ of $\pi_{\prt}(X)$ is conjugate to $[X]_{\prt'}$.  This follows immediately from the fact that 
    \[D_{\fp}(\pi_{\prt}(x), \pi_{\prt}(y)) = D_{\prt'}(x,y) \quad \text{for}\quad x,y\in X,\]
    which is clear from the definitions.
\end{proof}

\begin{example}\label{non-transitive_collections}
	Note that there can be nontransitive
	$\brt\subseteq(\mathcal P\ast \mathcal P)\backslash\mathcal P$
	such that $D_{\brt}$ is still a pseudometric. Let $\A =\{a,b,c\}$ and $X\subset \A^{\Z}$ be a subshift such that in every $x\in X$ the symbol $c$ appears with density zero. Let $\mathcal{P} = \{ [a]\mid a\in\A\}$ be the partition into letter cylinders. Then for the nontransitive set $\brt = \{ \{[a], [b]\}\}$ the function $D_{\brt}$ is clearly a pseudometric on $X$. 
    Yet, it is also clear that here one can enlarge $\brt$ to a transitive $\widetilde{\brt}\subseteq(\mathcal P\ast \mathcal P)\backslash\mathcal P$ (e.g.\ $\widetilde{\brt} = \fp$) to obtain the pseudometric $D_{\widetilde{\brt}}$ which is exactly the same as $D_{\brt}$.
    This motivates the following question: does there exist a subshift $X$ and a nontransitive $\brt\subseteq(\mathcal P\ast \mathcal P)\backslash\mathcal P$ such that $D_{\brt}$ is a pseudometric on $X$ and for any transitive 
	$\widetilde{\brt}\subseteq(\mathcal P\ast \mathcal P)\backslash\mathcal P$, $D_{\brt}$ and $D_{\widetilde{\brt}}$
	do not induce the same topology on $X$?
\end{example}

\subsection{Discrepancy substitution}\label{sec:discrepancysubstitution}

To capture and quantify how letters in a substitution separate under iteration, we introduce the notion of the \emph{discrepancy substitution}.

\begin{definition}\label{def:discrepancy_substitution_and_numbers}
	Let $\varphi\colon \mathcal{A}\to \mathcal{A}^*$ be a primitive substitution of  length $k\geq 2$. 
	Assume first that $\varphi$ has height 1.
	Let $\A_s = (\A\ast\A) \setminus \A$ be the set of unordered pairs of distinct letters of $\A$.
	We will often write the alphabet $\A_s$ in the following convenient form 
	\[\A_s =\left\{ \left.{a\choose b}\,\right|\, a,b\in\A,\ a\neq b \right\}
	\]
	with the convention that ${a\choose b} = {b\choose a}$.
	The \emph{discrepancy substitution} $\varphi_s$ of $\varphi$ is defined as follows: 
	 Write $\varphi(a)=a_0\dots a_{k-1}$ for each $a\in\A$.
	Then $\varphi_s\colon\A_s\to\A^*_s$ is given by\footnote{Here, the product corresponds to concatenation and if $\mathrm{I}_{a,b}=\emptyset$ then the product is the empty word.}
    
	\[
        \varphi_s{a\choose b} = \prod \left\{ \left.{a_i\choose b_i} \,\right|\, i\in \mathrm{I}_{a,b}\right\},\quad \text{where}\quad \mathrm{I}_{a,b} = \{0\leq i<k\mid a_i\neq b_i\}.
	\]
		If $\varphi$ has height $>1$, then its \emph{discrepancy substitution} is defined to be the discrepancy substitution of its pure base.
\end{definition}

\begin{definition}\label{def:discrepancyrate}
	Let $\varphi$ be a primitive substitution of constant length.
	The \emph{discrepancy rate} (resp.\ the \emph{discrepancy type}) of $\varphi$ is the growth rate (resp.\ the growth type) of its discrepancy substitution $\varphi_s$.
\end{definition}

	In other words, the discrepancy rate of $\varphi$ is the dominant eigenvalue of $\varphi_s$.
	In general, the discrepancy substitution does not  need to be primitive or of constant length; indeed it can very well happen that its letters do not exhibit uniform growth rates.
	Furthermore, if $\varphi$ is not  injective (i.e.\ $\varphi^n(a)=\varphi^n(b)$ for some $a\neq b$, $n\in\N$), then $\varphi_s$ will be erasing. 

    \begin{example}\label{ex:discrepancy_noninteger}
    Let 
     \(\varphi\) be a substitution of constant length \(3\) defined by
    \[
      a\mapsto aac,\quad b\mapsto acc,\quad c\mapsto aab;
    \]
    note that \(\varphi\) is primitive and of height one.
    The discrepancy substitution \(\varphi_s\) 
    is given by
    \[
      {a\choose b}\mapsto {a\choose c},\quad
      {a\choose c}\mapsto {b\choose c},\quad
      {b\choose c}\mapsto {a\choose c}\,{b\choose c}.
    \]
    and its incidence matrix is
    \[
      M_s=\begin{pmatrix}
        0 & 0 & 0\\
        1 & 0 & 1\\
        0 & 1 & 1
      \end{pmatrix}.
    \]
    It is already in normal primitive form:
    the top-left \(1\times1\) diagonal block is \(0\), and the lower-right
    \(2\times2\) diagonal block supported on
    \(\{{a\choose c},{b\choose c}\}\) is
    $\begin{pmatrix}0&1\\1&1\end{pmatrix}$;
    its dominant eigenvalue is the golden ratio
    $ \lambda_s=\frac{1+\sqrt{5}}{2}$.
    While the substitution $\varphi_s$ is not primitive, all letters have the same growth type $(\lambda_s,0)$.
\end{example}

\begin{example}\label{ex:discrepancysubstitution} Let $\varphi$ be a substitution of constant length 4 defined by:
\[a\mapsto baac,\quad b\mapsto bbca,\quad c\mapsto bcba;\] 
note that $\varphi$ is primitive and of height one.
	The discrepancy substitution $\varphi_s$ is given by
	\[{a\choose b}\to {a\choose b}{a\choose c}{a\choose c},\quad {a\choose c}\to {a\choose c}{a\choose b}{a\choose c},\quad {b\choose c}\to{b\choose c}{b\choose c}.
	\]
	The incidence matrix of $\varphi_s$ is given by
	\[M_{s} = \begin{pmatrix}
	1 & 2 & 0\\
	2 &1 & 0\\
	0 & 0& 2
	\end{pmatrix}.
	\]
	We see that $\varphi_s$ is in normal primitive form and it has two primitive components: one with the dominant eigenvalue 2 and the other one with dominant eigenvalue $3$. 
	Thus the discrepancy type of $\varphi$ is $(3,0)$; however not all pairs of distinct letters in $\A_s$ separate with the  rate  $\lambda_s = 3$.
\end{example}

\begin{example}\label{ex:sep_substitution_not_irreducible}
    Consider the substitution $\varphi$ given by
\[a\mapsto aaac,\quad b\mapsto abbb,\quad c\mapsto accb;\] 
note that $\varphi$ is primitive and of height one.
	The discrepancy substitution $\varphi_s$ is given by
	\[{a\choose b}\to {a\choose b}{a\choose b}{b\choose c},\quad {a\choose c}\to{a\choose c}{a\choose c}{b\choose c},\quad {b\choose c}\to {b\choose c}{b\choose c}.
	\]
	The incidence matrix of $\varphi_s$ is given by
	\[M_{s} = \begin{pmatrix}
	2& 0 & 0\\
	0 &2 & 0\\
	1 & 1& 2
	\end{pmatrix}.
	\]
    It is in normal primitive form and has 3 primitive components, all with the dominant eigenvalue 2. We see that the letters ${a\choose b}$, ${b\choose c }$, and ${a\choose c}$ have the same growth rate 2, but not the same growth type: ${a\choose b}$ and ${b\choose c }$ have growth type (2,1), and ${a\choose c}$ has growth type (2,0). The discrepancy type of $\varphi$ is (2,1).
\end{example}

The discrepancy substitution $\varphi_s$ works well with taking powers of $\varphi$, and, in the case $\varphi$ has height 1, keeps track of the number of indices on which the words $\varphi^n(a)$ and $\varphi^n(b)$ differ.

\begin{lemma}\label{lem:substitution_counts_differences}
Let $\varphi$ be a primitive substitution of length $k\geq 2$, $\varphi'$ its pure base, and $\varphi_s$ its discrepancy substitution.
 Let $n\in\N$.\begin{enumerate}
\item\label{lem:substitution_counts_differences1} The discrepancy substitution of $\varphi^n$ is equal to $\varphi_s^n$.
\item\label{lem:substitution_counts_differences2} For all $\{a,b\}\in\mathcal{A}_s$ we have
\[\abs{\varphi_s^n{a \choose b}} = \abs{\{0\leq i<k^n\mid a_i\neq b_i\}},\]
	where $(\varphi')^n(a) = a_0\dots a_{k^{n}-1}$ and 
$(\varphi')^n(b) = b_0\dots b_{k^{n}-1}$.
\end{enumerate}
\end{lemma}
\begin{proof}
	Claim \eqref{lem:substitution_counts_differences2} follows directly from \eqref{lem:substitution_counts_differences1} and the definition of the discrepancy substitution.
	
	It is straightforward to see that the pure base of the substitution $\varphi^n$ is equal to $(\varphi')^n$.
	Hence, to show \eqref{lem:substitution_counts_differences1}, we may assume without loss of generality that $\varphi$ has height 1, i.e.\ $\varphi=\varphi'$.
	Note that it is enough to show that for two substitutions $\varphi,\ \tau\colon \A\to\A^*$ of  length $k$, we have $(\varphi\circ\tau)_s = \varphi_s\circ\tau_s$.
	This follows from the fact that for any distinct $a,b\in\A$ and $0\leq i<k^2$, $i=i'+jk$ with $0\leq i',j<k$, 
	$((\varphi\circ\tau)(a))_i\neq ((\varphi\circ\tau)(b))_i$ if and only if $\tau(a)_j\neq \tau(b)_j$ and $\varphi(\tau(a)_j)_{i'}\neq \varphi(\tau(b)_j)_{i'}$.
\end{proof}

	 The following proposition gathers the basic properties of the numbers $\lambda_s$: the extreme values of discrepancy rates correspond to the concrete dynamical properties of the underlying substitution systems.
	 
	  \begin{proposition}\label{prop:sepnumberproperties}
	  Let $\varphi$ be a primitive substitution of  length $k\geq 2$ and let $\lambda_s$ be its discrepancy rate.
      The following are true:
\begin{enumerate}
\item\label{prop:sepnumberproperties1} $\lambda_s = 0$ or $1\leq \lambda_s \leq k$,
\item\label{prop:sepnumberproperties2} $\lambda_s = 0$ if and only if $X_{\varphi}$ is finite,
\item\label{prop:sepnumberproperties3} $\lambda_s =k$ if and only if $X_{\varphi}$ does not have discrete spectrum if and only if it is not mean equicontinuous.
\end{enumerate}
\end{proposition}
\begin{proof}
	The first claim follows from the fact that $\lambda_s$ is the dominant eigenvalue of a nonnegative integer matrix. 
	
	For the second claim note first that we may assume that $\varphi$ has height $h=1$ by passing to the pure base of $\varphi$. In this case $\lambda_s=0$ if and only if the matrix $M_{\varphi}$  is nilpotent, i.e.\ $M^n_{\varphi}=0$ for some $n\in\N$. 
	This is equivalent to the fact that $\varphi_s^n(\alpha)=\varepsilon$  for all $\alpha\in\A_s$. 
	This, in turn, is  equivalent to the fact that the words $\varphi^n(a)$, $a\in\A$ are all equal.
	This clearly implies that $X_{\varphi}$ is finite.
    
    For the other direction, assume that $X_{\varphi}$ is finite; we may assume without loss of generality that $\varphi$ has some fixed point $x$. 
    Since $X_{\varphi}$ is minimal we have $X_{\varphi}=\Z/m\Z $ for some $m\geq 1$ and 
    \[ x = \dots uuu.uuu\dots\]
    for some $u$ of length $m$.
    By Remark \ref{rem:height}, the height $h$ is the biggest $n$ coprime with $k$ such that $\Z/n\Z$ is a factor of $X_{\varphi}$. 
    Since $h=1$ by assumption, all prime divisors of $m$ must divide $k$. 
    Hence, $m$ divides $k^n$ for some $n$ big enough.
    But this means that $\varphi^n(a) = u^{q}$ for all $a\in \A$, where $q= k^n/m$; in particular, all  $\varphi^n(a)$, $a\in\A$ are the same.

	For the second equivalence in \eqref{prop:sepnumberproperties3} see \cite[Cor. 5.2]{FuhrmannGroeger2020}.
	For the first equivalence, we may assume without loss of generality that $\varphi$ has height 1.
	By Theorem \ref{thm:dekking}, $\varphi$ has discrete spectrum if and only if it has a coincidence.
	 If $\varphi$ has a coincidence, then  for some $n\geq 1$, $\abs{\varphi^n_s(\alpha)}\leq k^n-1$ for all pairs of distinct letters $\alpha=\{a,b\}$, and so, $\lambda^n_s\leq k^n-1<k^n$, and thus $\lambda_s<k$.
	For the other direction, assume that $\lambda_s<k$. 
	This  implies that for each pair of distinct letters $\alpha$, $\abs{\varphi_s^{n_{\alpha}}(\alpha)}\leq k^{n_{\alpha}}-1$ for some $n_{\alpha}\geq 1$.
	Putting $n = \max n_{\alpha}$, we have that  $\abs{\varphi_s^{n}(\alpha)}\leq k^{n}-1$ for all $\alpha$.

	In particular, for any $\mathcal{B}\subset \A$ of size at least 2, there exists $0\leq i<k$ such that $\abs{\varphi^n_i(\mathcal{B})}<\abs{\mathcal{B}}$.
	Using this inductively, starting with $\mathcal{B}=\A$ we get that $\varphi^n$ (and thus $\varphi$) has a coincidence.
	\end{proof}

	We will see in the next sections that, for the values between 1 and $k$, the discrepancy rates allow us to compute the amorphic complexity of $X_{\varphi}$  and that the  smallest possible value $\lambda_s=1$ of the discrepancy rate of an infinite system $X_{\varphi}$ corresponds to tameness and, equivalently in this case, nullness of $X_{\varphi}$.

\subsection{Pseudometrics related to pairs of letters of maximal growth}\label{pseudometricforsubstitutions}

For our purposes the most important pseudometric on $X_{\varphi}$ will be given by $D_{\mg}$ with $\mg$ being the set of all distinct pairs of (cylinders of) letters that have maximal growth with respect to the discrepancy substitution $\varphi_s$.

\begin{definition}\label{def:discrepancypseudometric}
	Let $\varphi$ be a substitution of  length $k\geq 2$ and height 1. 
    We will identify each letter $a\in\A$ with its natural cylinder set $[a] = \{x\in X_{\varphi}\mid x_0=a\}$ and the alphabet set $\A$ with the partition $P=\{[a]\mid a\in \A\}$ by the letter cylinders.
	We will denote by $\mg\subset \A_s = (\A\ast\A) \setminus \A$ the set of  pairs of letters of maximal growth with respect to $\varphi_s$ and by $D_{\mg}$ the corresponding pseudometric on $X_{\varphi}$.
\end{definition}

Lemma \ref{lem:when_d_D_is_pseudometric} below shows that $D_{\mg}$ is indeed a pseudometric on $X_{\varphi}$.

\begin{continueexample}{ex:discrepancysubstitution}
    Let $\varphi$ be a substitution 
\[a\mapsto baac,\quad b\mapsto bbca,\quad c\mapsto bcba;\] 
with its discrepancy substitution $\varphi_s$:
	\[{a\choose b}\to {a\choose b}{a\choose c}{a\choose c},\quad {a\choose c}\to {a\choose c}{a\choose b}{a\choose c},\quad {b\choose c}\to{b\choose c}{b\choose c}.
	\]
    Here the set of letters of maximal growth w.r.t.\ $\varphi_s$ is given by $\mg = \{{a\choose b}, \,{a\choose c}\}$ (remembering that  ${a\choose b}$ really denotes the set $\{a,b\}$) and the corresponding pseudometric is given by
    \[
    D_{\mg}(x,y) = \limsup\limits_{n\to\infty}\frac{1}{2n+1}
        \abs{\left\{ -n\leq i \leq n \,\Big| \, {x_i \choose y_i} \in \mg \right\}}.
    \]
\end{continueexample}

\begin{lemma}\label{lem:when_d_D_is_pseudometric} 
Let $\varphi$ be a primitive substitution of length $k\geq 2$ and height 1. Then $D_{\mg}$ is a pseudometric on $X_{\varphi}$.
\end{lemma}
\begin{proof}
By Lemma \ref{lem:pseudometricversusfactormap}, it is enough to show that $\mg$ is transitive.

	Suppose that the claim does not hold and that there exist $\{a,b\}\in \mg$ and $c\in\mathcal{A}$ such that both $\{a,c\}$ and $\{b,c\}$ do not lie in $\mg$. Clearly, $c\neq b$ and $c\neq a$. Write 
\[\mathcal{A}_s=M\sqcup M' \sqcup L,\]
where
\begin{enumerate}
\item $M$ is the set of letters of maximal growth $\lambda$ which appear in $\varphi_s^n{a \choose  b}$ for some $n\geq 0$,
\item $M'$ is the (possibly empty) set of letters of maximal growth which do not appear in $M$,
\item $L$ is the  set of letters with growth strictly less than $\lambda$.
\end{enumerate}
After permuting the matrix if necessary we can write $M_{\varphi_s}$ in the upper block-diagonal form:
\[
M_{\varphi_{s}}=\begin{bmatrix}
       P_1 & 0&  0\\
        *& P_2 & 0  \\
       * & *&  P_3\\
    \end{bmatrix},
\]
 where each matrix $P_i$ corresponds to the sets $M$, $M'$ and $L$, respectively. Note that the maximal growth rate of $\varphi_s$ is equal to the dominant eigenvalue $\rho(P_1)$ of $P_1$ and all letters in $L$ have growth rate $\leq$ than $\rho(P_3)<\rho(P_1)$. 

Let $N\subset \mathcal{A}_s$ be the following set of letters:
\[N=\left\{{\varphi^n(b)_i\choose \varphi^n(c)_i}\mid i\geq 1,\ {\varphi^n(a)_i \choose \varphi^n(b)_i}\in M\right\}.\] Since $\{b,c\}$ does not have maximal growth, $N\subset L$. Since neither  $\{a,c\}$ nor $\{b,c\}$ lie in $\mg$, for each $\{d,e\}\in M$, there exists $\{d',e'\}\in N$ (with either $d'=d$ or $e'=e$) such that
\begin{equation}\label{eq:more_of_slow_letters} 
\abs{\varphi_s^n {d'\choose e'}}_{N} \geq \abs{\varphi_s^n{d\choose e}}_{M}
\end{equation} for all $n\geq 0$ (here, $\abs{w}_B$ denotes the number of occurrences of letters from $B$ in the word $w$). We let $||\cdot||_1$ denote the first matrix norm which is  the maximum absolute column sum of the matrix. We now see that 
\[\rho(P_1) = \lim_{n\to\infty} \left(||P_1^n||_1\right)^{1/n}\leq \lim_{n\to\infty} \left(||P_3^n||_1\right)^{1/n}=\rho(P_3),
\]
using \eqref{eq:more_of_slow_letters}, the fact that $N\subset L$ and the fact that $M_{\varphi_s}$ is nonnegative. This contradicts the fact that $\{b,c\}$ is not of maximal growth.
\end{proof}

	Let $\varphi$ be a primitive substitution of  length $k$.
	 It is easy to see that for all $x,y\in X_{\varphi}$, $D_{\fp}(x,y)=0$ if and only if $D_{\fp}(\varphi(x),\varphi(y))=0$, and, thus, the substitution $\varphi$ preserves the equivalence classes of $[X_{\varphi}]$ and induces a well defined map $[\varphi]$ on $[X_{\varphi}]$, see \cite[p.\ 11]{FuhrmannGroeger2020}.
	The same is true for the pseudometric $D_{\mg}$.

\begin{lemma}\label{lem:proportion_of_ourmetric_to_Besicovitch}
Let $\varphi\colon \A \to \A^*$ be a primitive substitution of  length $k\geq 2$ and height 1.  Then
\begin{enumerate}
\item\label{lem:proportion_of_ourmetric_to_Besicovitch1} $D_{\mg}(x,y)=0$ if and only if $D_{\mg}(\varphi(x),\varphi(y))=0$ for any $x,y\in X_{\varphi}$,
\item\label{lem:proportion_of_ourmetric_to_Besicovitch2} for any $x,y\in X_{\varphi}$ with $D_{\fp}(x,y)>0$ and $n$ big enough we have
\begin{equation}\label{eq:proportion_of_ourmetric_to_Besicovitch}
\frac{D_{\mg}(\varphi^n(x),\varphi^n(y))}{D_{\fp}(\varphi^n(x),\varphi^n(y))}\geq
\frac{D_{\mg}(x,y)}{D_{\fp}(x,y)}.
\end{equation} 
\end{enumerate}
\end{lemma}
\begin{proof}
    Property \eqref{lem:proportion_of_ourmetric_to_Besicovitch1} is clear, since for any distinct pair of letters $\alpha={a\choose b}$ of not maximal growth (i.e.\ lying in $\A_s\setminus \mg$), all (pairs of) letters appearing in $\varphi_s(\alpha)$ are again of not maximal growth and for any  distinct pair of letters $\alpha={a\choose b}$ of  maximal growth, $\varphi(\alpha)$ contains at least one appearance of a (pair of) letters of maximal growth.

    By \eqref{lem:proportion_of_ourmetric_to_Besicovitch1}, for any $n\geq 1$, $D_{\mg}(x,y)=0$ if and only if $D_{\mg}(\varphi^n(x),\varphi^n(y))=0$, in which case \eqref{lem:proportion_of_ourmetric_to_Besicovitch2} is clear.
        Thus, we may assume without loss of generality that $D_{\mg}(\varphi^n(x),\varphi^n(y))>0$ for all $n\geq 0$.
    Let $\lambda$ be the maximal growth rate of $\varphi_s$, and let $\gamma<\lambda$ be strictly greater than all nonmaximal growth rates of $\varphi_s$. Let $C>0$ be such that $\abs{\varphi_s^n(\alpha)}\geq C\lambda^n$ for each $\alpha\in \mg$.
    Using Lemma \ref{lem:substitution_counts_differences}, for each $n\geq 1$ and $i\geq 0$ we can write
\[ \frac{D_{\fp}(\varphi^n(x),\varphi^n(y))}{D_{\mg}(\varphi^n(x),\varphi^n(y))}=\lim_{i\to\infty} \frac{t^{(n)}_{i}}{s^{(n)}_{i}},\] 
where
\[t^{(n)}_{i}=  \sum_{\abs{j}<i,\, x_j\neq y_j} \abs{\varphi_s^n{x_j\choose y_j}}\quad \text{and} \quad 
s^{(n)}_{i}=\sum_{\abs{j}<i, \, {x_j \choose y_j}\in \mg} \abs{\varphi_s^n{x_j \choose y_j}}.\]
    Write $\alpha_j = {x_j \choose y_j}$. 
    Let $n$ be big enough so that $\abs{\varphi_s^n(\alpha)}<\gamma^n$ for any $\alpha\in\A_s$ of nonmaximal growth and $\gamma^n/C\lambda^n\leq 1$. For such $n$, rewriting each term $t_{i}/s_{i}$, we get that
\begin{equation*} 
\begin{split}
\frac{t^{(n)}_{i}}{s^{(n)}_{i}} &= 1 +  \left(\sum_{\abs{j}<i, \alpha_j\in \mathcal{A}_s\setminus \mg} \abs{\varphi_s^n(\alpha_j)}\right)\cdot\left( \sum_{\abs{j}<i, \alpha_j\in \mg} \abs{\varphi_s^n(\alpha_j)}\right)^{-1} \\
&\leq 1+ \gamma^n\cdot \left(\abs{\{\abs{j}<i\mid \alpha_j\in \mathcal{A}_s\setminus \mg\}}\right)\cdot(C\lambda^n)^{-1}\cdot  \left(\abs{\{\abs{j}<i\mid \alpha_j\in  \mg\}}\right)^{-1} \\
& \leq   \abs{\{\abs{j}<i\mid x_j\neq y_j\}}\cdot  \abs{\{\abs{j}<i\mid \alpha_j\in  \mg\}}^{-1} = t^{(0)}_{i}(s^{(0)}_{i})^{-1}
\end{split}
\end{equation*}
Letting $i$ go to infinity gives the claim. 
\end{proof}

\section{Amorphic complexity of constant length substitution shifts}\label{sec:complexityofsubstitution}

This section is devoted to the proof of the following result.

\begin{theorem}\label{thm:ac_of_automatic_systems} Let $\varphi$ be a primitive substitution of  length $k\geq 2$.  The amorphic complexity of $X_{\varphi}$ is given by
\[\ac(X_{\varphi})=\frac{\log k}{\log k - \log \lambda_s},\]
where $\lambda_s$ is the discrepancy rate of $\varphi$.
\end{theorem}

The following example shows that passing to the pure base before computing the discrepancy rate is necessary in Theorem \ref{thm:ac_of_automatic_systems}.

\begin{continueexample}{ex:heightandpurebase} 
	Let
	\[\varphi(0)=010,\quad \varphi(1)=102,\quad \varphi(2)=201\]
    be a substitution of height 2 with its pure base $\varphi'$ given by
	\[\varphi'(a)=aab,\quad \varphi'(b)=aba.\]
	Note that $\varphi'$ has a coincidence and so, by Theorem \ref{thm:dekking}, $X_{\varphi}$ has discrete spectrum.
	Furthermore, the amorphic complexity of $X_{\varphi}$ and $X_{\varphi'}$ coincide. 
	However, if one were to construct the discrepancy substitution directly from $\varphi$ (as opposed from its pure base $\varphi'$) one would get the dominant eigenvalue of the discrepancy substitution equal to $3$ while the dominant eigenvalue of the discrepancy substitution of the pure base $\varphi'$ is $2$. Thus, the claim of Theorem \ref{thm:ac_of_automatic_systems} does not hold for $\varphi$ before purification.
\end{continueexample}

\begin{remark}\label{rem:Baakeresult}
    In \cite{BaakeGaehlerGohlke2025}, the authors introduce the notion of a \emph{discrepancy inflation} for a self-similar, nonperiodic topological tiling dynamical system $(\mathbb{X},\mathbb{R}^d)$ with pure-point spectrum generated by a primitive inflation, using it to establish lower and upper bounds for the amorphic complexity of $(\mathbb{X},\mathbb{R}^d)$. Though their work focuses on tilings (rather than the symbolic setting), it is easy to see that, for a substitution $\varphi$ of constant length, their discrepancy inflation corresponds to our discrepancy substitution.

    By translating their main result (\cite[Thm.\ 4.13 \& Prop.\ 4.14]{BaakeGaehlerGohlke2025}) to our framework, we obtain the following bounds:
    \[
        \frac{\log k}{\log k - \log \lambda_{min}} \leq \oac(X_{\varphi}) \leq \frac{\log k}{\log k - \log \lambda_s},
    \]
    where $\lambda_s$ (resp.\ $\lambda_{min}$) denotes the maximal (resp.\ minimal) growth rate of letters in the discrepancy substitution $\varphi_s$, and $k$ is the length of the substitution. If $\lambda_s = \lambda_{min}$, this yields a closed-form expression for the amorphic complexity of $X_{\varphi}$. However, in general, one typically expects that $\lambda_{min} < \lambda_s$; see e.g.\ Example~\ref{ex:sep_substitution_not_irreducible}.
\end{remark}

In the proof we will need the following fact from \cite{FuhrmannGroeger2020}.

 \begin{lemma}\label{lem:fractalestimates}
     Let $\varphi$ be a primitive substitution of length $k\geq 2$ and of height 1. Assume that $X_{\varphi}$ is infinite. If for  some constants $A,B >0$,
     \[A\cdot D_{\fp}(x,y)\leq D_{\fp}(\varphi(x),\varphi(y))\leq B\cdot D_{\fp}(x,y) \quad \text{for all}\quad x,y\in X_{\varphi}, 
     \] 
     then
     \[ \frac{\log k}{-\log A} \leq \underline\Dim_{B}[X] \leq \overline\Dim_{B}[X]\leq \frac{\log k}{-\log B}. \]
      \end{lemma}
\begin{proof}
    This follows from \cite[Thm.\ 7.5 \& Prop. 7.7]{FuhrmannGroeger2020}, compare also with the proof of Theorem 7.6. in \cite{FuhrmannGroeger2020}. 
\end{proof}

To prove Theorem \ref{thm:ac_of_automatic_systems}, we first reduce to the pure base (height $1$); the extreme cases $\lambda_s\in\{0,k\}$ are treated by  Proposition \ref{prop:sepnumberproperties}.  
To establish the formula $\ac(X_{\varphi}) = \log k \, \bigl(\log k - \log \lambda_s\bigr)^{-1}$
we apply Lemma \ref{lem:fractalestimates} to the powers $\varphi^n$ of $\varphi$; this requires sufficiently sharp estimates of the form
\[
A_n \leq 
\frac{D_{\fp}(\varphi^n(x),\varphi^n(y))}{D_{\fp}(x,y)} 
\leq B_n .
\]
The upper bounds are relatively straightforward and follow from the fact that all letters of the discrepancy substitution $\varphi_s$ have growth rates uniformly bounded by $\lambda_s$. The main difficulty is to obtain lower bounds that are sharp enough; for this one has to show that— in an appropriately quantitative sense — pairs of letters $\alpha={a\choose b}$ that are \emph{not} of maximal growth with respect to discrepancy substitution $\varphi_s$ do not contribute to the asymptotics.  
The core geometric idea is therefore to replace the Besicovitch pseudometric $D_{\fp}$ by the seemingly coarser pseudometric $D_{\mg}$, built only from pairs of letters of maximal growth, and to show that these two pseudometrics are in fact Lipschitz equivalent on $X_{\varphi}$ (Proposition \ref{prop:metricsequivalnet}). Since $D_{\mg}$ detects only pairs $\alpha={a\choose b}$ with growth rate $\lambda_s$, the derivation of the lower bounds then proceeds in complete analogy with the upper-bound estimates, yielding the claimed formula for the amorphic complexity.

\begin{proposition}\label{prop:metricsequivalnet}
    Let $\varphi$ be a primitive substitution of  length $k\geq 2$ and height $1$. Assume that $X_{\varphi}$ is infinite and has discrete spectrum.
    The  Besicovitch pseudometric $D_{\fp}$ and the pseudometric $D_{\mg}$ are Lipschitz equivalent on $X_{\varphi}$.
\end{proposition}
\begin{proof}
    We may assume without loss of generality that $\varphi$ has a fixed point $x$. 
    By Theorem \ref{thm:dekking}, the maximal equicontinuous factor of $X_{\varphi}$ is given by $\Z_k$ and is conjugate to $[X_{\varphi}]$, since $X_{\varphi}$ is mean equicontinuous.  Let $\mg\subset \mathcal{A}_s$ be the set of letters of maximal growth with respect to $\varphi_s$ and let $\prt = \mg' = \A_s\setminus \mg$ be the complement of $\mg$.
    Let $\pi = \pi_{\prt}$ be the factor map from $X_{\varphi}$ onto the subshift $\pi_{\prt}(X_{\varphi})\subset \left(\A_s/\sim_{\prt}\right)^{\Z}$ defined in \eqref{eq:factormapfrompartition} in Section \ref{sec:pseudometrics}.
    
    It is not hard to see that $Y=\pi(X_{\varphi})$ is infinite: suppose for contradiction that $Y$ is finite. Then $Y$ itself is a minimal equicontinuous system which is a factor of $\Z_k$ and so $Y = \Z/m\Z$ with $m=\prod q^{s_{q}}$, where $q$'s are the primes dividing $k$ and $s_q\geq 0$.
    Since $\pi$ is a factor map, we have
    \begin{equation}\label{eq:whynotfinite}
        \pi(x) = \pi(T^{nm}x)\quad \text{for all}\quad n\in\Z.
    \end{equation}
    Clearly, there exists $c\in Z$ such that $\pi(x)\neq \pi(T^cx)$---it is enough to shift $x$ until we see a pair $\alpha = {a\choose b}$ of maximal growth (i.e.\ lying in $\mg$) in ${x \choose T^cx}$.
    Let $n$ be big enough, so that $m$ divides $k^n$.
    Since for any pair of maximal growth $\alpha$, $\varphi(\alpha)$ has to contain some other pair $\alpha'$ of maximal growth, we have
    \[\pi(x) = \pi(\varphi^n(x)) \neq \pi(\varphi^n(T^cx)) = \pi(T^{ck^n}x).
    \]
    This contradicts \eqref{eq:whynotfinite}, since $m$ divides $k^n$.

    Since $\pi(X_{\varphi})$ is infinite, by Theorem \ref{thm:reemclemens}, it still has $\Z_k$ as the maximal equicontinuous factor. 
    By Lemma \ref{lem:pseudometricversusfactormap}, $[X]_{\mg}$ is the $\mathrm{MEF}$ of $\pi(X_{\varphi})$ and hence it is $\Z_k$.
    This means that the natural factor map $[X_{\varphi}]\to[X_{\varphi}]_{\mg}$ is in fact a selfmap from the equicontinuous system $\Z_k$ to itself, and so, it is a conjugacy. In particular, the pseudometrics $D_{\fp}$ and $D_{\mg}$ are equivalent  and $D_{\fp}(x,y)=0$ if and only if $D_{\mg}(x,y)=0$.
    
    We will now show that $D_{\fp}$ and $D_{\mg}$ are furthermore Lipschitz equivalent, that is, the map $[X_{\varphi}]\to[X_{\varphi}]_{\mg}$ is bi-Lipchitz. 
Taking a power of $\varphi$ if necessary, we can assume that \eqref{eq:proportion_of_ourmetric_to_Besicovitch} in Lemma \ref{lem:proportion_of_ourmetric_to_Besicovitch} holds for all $n\geq 1$.
    It is clear that $D_{\mg}\leq D_{\fp}$. 
    Since $D_{\mg}$ and $D_{\fp}$ are equivalent, we can write 
\[[X_{\varphi}]=\coprod_{0\leq i<k} [T^{i}\varphi(X_{\varphi})], \]
with $[T^i\varphi(X)]$  pairwise disjoint and closed in $D_{\mg}$ metric, see also \cite[Lem.\ 7.2]{FuhrmannGroeger2020}. 
    Thus, there exists some $M>0$ such that $D_{\mg}(x,y)\geq M\geq M D_{\fp}(x,y)$ for all $x,y\in [X_{\varphi}]$ lying in different sets $[T^i\varphi(X)]$. 
    Now, let $x,y\in [X_{\varphi}]$ be distinct points lying in the same set $[T^i\varphi(X_{\varphi})]$.  Since $x$ and $y$ are different, there exist some $n\geq 1$, $0\leq c_n<k^n$ and $x',y'$ lying in different sets $[T^i\varphi(X_{\varphi})]$ such that $x=T^{c_{n}}\varphi^n(x')$ and $y=T^{c_{n}}\varphi^n(y')$. By Lemma \ref{lem:proportion_of_ourmetric_to_Besicovitch},
\[\frac{D_{\mg}(T^{c_{n}}\varphi^n(x'),T^{c_{n}}\varphi^n(y'))}{D_{\fp}(T^{c_{n}}\varphi^n(x'),T^{c_{n}}\varphi^n(y'))} =
\frac{D_{\mg}(\varphi^n(x'),\varphi^n(y'))}{D_{\fp}(\varphi^n(x'),\varphi^n(y'))}
\geq \frac{D_{\mg}(x',y')}{D_{\fp}(x',y')}\geq M,\]
which finishes the proof.
\end{proof}

\begin{proof}[Proof of Theorem \ref{thm:ac_of_automatic_systems}]

Let $X_{\varphi}$ be a system generated by a primitive substitution of length $k$ and  height  $h$. 
    By Proposition \ref{prop:sepnumberproperties}, we may assume that $X_{\varphi}$ is infinite and mean equicontinuous; in this case $0<\lambda_s<k$.
	By Theorem \ref{thm:dekking}, there exists a substitution $\varphi'$ of length $k$ and height 1 such that $X_{\varphi}$ is conjugate to the $h$-height suspension of $X_{\varphi'}$, and so $\ac(X_{\varphi})=\ac(X_{\varphi'})$, assuming that it exists.
    Clearly, $X_{\varphi'}$ is again inifinite and mean equicontinous.
    Hence, we may assume without loss of generality that the original substitution $\varphi$ is of height 1.
    Furthermore, by taking a power of $\varphi$ if necessary, we may assume that $\varphi_s$ is in normal primitive form.
    
Let $(\lambda_s,d)$ be the discrepancy type of $\varphi$; in particular, the dominant eigenvalue of the discrepancy substitution $\varphi_s$ is $\lambda_s$. Let $\mg\subset \mathcal{A}_s$ be the set of letters of maximal growth (with respect to $\varphi_s$); by Lemma \ref{lem:growth_of_letters}, 
\begin{equation}\label{eq:bounds_on_letter_growth} \abs{\varphi_s^n(\alpha)}\leq Nn^d\lambda_s^n  \text{ for } \alpha\in \A_s \quad \text{\&}\quad         N'\lambda_s^n\leq \abs{\varphi_s^n(\alpha)}\text{ for } \alpha\in \mg 
\end{equation} for some constants $N,N'>0$ and all $n\geq 1$.

First we show that for some constants $C, C'>0$:
\begin{equation}\label{eq:IFS}
\frac{C'\lambda_s^n}{k^n}\leq \frac{D_{\fp}(\varphi^n(x),\varphi^n(y))}{D_{\fp}(x,y)}\leq \frac{Cn^d\lambda_s^n}{k^n}\quad \text{for}\quad x,y\in[X]\, \text{ and }\, n\geq 1.
\end{equation}
    For fixed $n\geq 1$ we write 
\begin{equation}\label{eq:IFS_ineq}D_{\fp}(\varphi^n(x),\varphi^n(y))=\lim_{i\to\infty} \frac{\abs{\{j< ik^n\mid \varphi^n(x)_j\neq \varphi^n(y)_j\}}}{ik^n}=
\lim_{i\to\infty} \frac{\sum_{j<i,\, \alpha_j\in \A_s} \abs{\varphi_s^n(\alpha_j)}}{ik^n},
\end{equation} 
where $\alpha_j = {x_j\choose y_j}$; note that the use of the one-sided averaging and the passage to the subsequence $k^n\N$ in \eqref{eq:IFS_ineq} are justified by Remark \ref{rem:folner}.

Since all letters $\alpha\in\A_s$ have the growth type at most $(\lambda_s,d)$, by \eqref{eq:bounds_on_letter_growth}, we have that
\[ D_{\fp}(\varphi^n(x),\varphi^n(y))\leq \lim_{i\to\infty}  \frac{Nn^d\lambda_s^n}{ik^n}\abs{\{j<i\mid x_j\neq y_j\}}= \frac{Nn^d\lambda_s^n}{k^n}D_{\fp}(x,y),\]
which shows the second inequality in \eqref{eq:IFS} with $C=N$.

Ignoring in \eqref{eq:IFS_ineq} all letters $\alpha_j={x_j\choose y_j}\in \A_s\setminus \mg$, we have
\begin{equation}\label{eq:lowerbound1}D_{\fp}(\varphi^n(x),\varphi^n(y))\geq\lim_{i\to\infty}  \frac{1}{ik^n}\sum_{j<i, \alpha_j\in \mg} \abs{\varphi_s^n(\alpha_j)} \geq \lim_{i\to\infty}  \frac{N'\lambda_s^n}{ik^n} \abs{\{j<i\mid \alpha_j\in \mg\}}, 
\end{equation}
 using the lower bound in \eqref{eq:bounds_on_letter_growth} for letters in $\mg$.
    By Proposition \ref{prop:metricsequivalnet}, the pseudometrics $D_{\fp}$ and $D_{\mg}$ are Lipschitz equivalent, and thus there is $M>0$ such that $D_{\mg}(x,y)\geq M D_{\fp}(x,y)$ for all $x,y\in X_{\varphi}$. Hence, continuing \eqref{eq:lowerbound1}, we have  \[D_{\fp}(\varphi^n(x),\varphi^n(y))\geq \frac{N'\lambda_s^n}{k^n} D_{\mg}(x,y)\geq \frac{N'M\lambda_s^n}{k^n} D_{\fp}(x,y).\]
 This shows the first inequality in \eqref{eq:IFS} with $C'=N'M$.

For each $n\geq 1$ let 
\[A_n=\frac{C'\lambda_s^n}{k^n}\quad \text{and}\quad B_n=\frac{Cn^d\lambda_s^n}{k^n}.\]
The systems generated by $\varphi^n$ are the same for all $n\geq 1$; in particular, $\overline{\ac}(X_{\varphi})=\overline{\ac}(X_{\varphi^{n}})$ (resp.\ $\underline{\ac}(X_{\varphi})=\underline{\ac}(X_{\varphi^{n}})$). Since $\varphi^n$ is of length $k^n$, by Lemma \ref{lem:fractalestimates} applied to each $\varphi^n$, we get that
\[\frac{\log k^n}{- \log A_n}\leq\underline{\ac}(X)\leq \overline{\ac}(X)\leq\frac{\log k^n}{- \log B_n}\quad \text{for}\quad n\geq 1.\]
Letting $n$ go to infinity, we see that the amorphic complexity of $X$ exists and equals $\ac(x)=\frac{\log k}{\log k - \log \lambda_{s}}.$
\end{proof}

\begin{continueexample}{ex:discrepancysubstitution}
    Since the Besicovitch spaces of $X_{\varphi}$ and $\pi_{\mg'}(X_{\varphi})$ are always Lipschitz equivalent, one could ask if the systems $X_{\varphi}$ and $\pi_{\mg'}(X_{\varphi})$ have to be conjugate. This is, in general, not true: $\pi_{\mg'}(X_{\varphi})$ can be a proper factor of $X_{\varphi}$. 
    Recall the substitution  $\varphi\colon \A\to\A^*$ given by 
\[a\mapsto baac,\quad b\mapsto bbca,\quad c\mapsto bcba\] 
and the fact that the set of letters of maximal growth w.r.t.\ $\varphi_s$ is given here by $\mg = \{{a\choose b}, \,{a\choose c}\}$.
    Thus $\mg'=\{{b\choose c}\} \cup \A$ and the factor map $\pi=\pi_{\mg'}$ is a coding identifying letters $b$ and $c$. One sees easily that the factor map $\pi$ is not bijective: two different points
    \[
    x = \dots\varphi^2(b)\varphi(b).b\varphi(ca)\varphi^2(ca)\dots
    \] and
    \[
    y = \dots \varphi^2(b)\varphi(b).c\varphi(ba)\varphi^2(ca)\dots
    \] lying in $X_{\varphi}$ have the same image under $\pi$.
\end{continueexample}

By \cite{FuhrmannGroeger2020}, the amorphic complexity of a minimal automatic system is
either zero or lies in $[1, \infty]$ (this follows now also from Theorem \ref{thm:ac_of_automatic_systems} and Proposition \ref{prop:sepnumberproperties}). 
It is not hard to see that in this class of systems one can realise a dense set of values in $[1,\infty)$ for amorphic complexity.

\begin{proposition}\label{prop:denseness}
 Let $\A$ be an alphabet of size at least 2. In the class of minimal automatic systems over $\A$, amorphic complexity takes a dense set of values in $[1,\infty)$.   
\end{proposition}
\begin{proof}
    Fix $k\geq 2$ and let $t=\abs{A}! +\abs{A}$.
       It is not hard to see that for any  $n$ such that $k^n\geq t$  and any $1\leq l<k^n$ there exists a primitive substitution $\varphi$ of length $k^n$ and height $1$ such that $l$ of the column sets of $\varphi$ have size $\abs{\A}$ (i.e.\ are bijective) and the rest of $k^n-l$ column sets have size 1 (i.e.\ form coincidences). Here, the requirement that the length $k^n$ of the substitution is  $\geq \abs{A}! +\abs{A}$ means that either $l\geq\abs{A}!$ or $k^n-l\geq\abs{A}$, which lets us easily guarantee that we can find a \emph{primitive} substitution satisfying all the above.

   It is clear that for such $\varphi$ its discrepancy substitution $\varphi_s$ is of constant length $l$ and so its discrepancy rate is $\lambda_s=l$.
Thus, the values of amorphic complexity for these substitutions form the set
\[
\mathcal{S}
=
\left\{
\frac{n\log k}{\,n\log k - \log l\,}
\mid n \ge \log_k t,\  1\leq l < k^n \right\}.
\]
Putting $r := \log l/(n\log k)$, we can write
$\frac{n\log k}{n\log k - \log l} = \frac{1}{1-r}$;
note that  $0 \le r < 1$ (since $1\leq l < k^n$). Since the map
$[0,1)\to[1,\infty),$ $r \longmapsto \frac{1}{1-r}$
is a homeomorphism, it suffices to show that the set
\[
\mathcal{S}'
:=
\left\{
\frac{\log l}{n\log k}
\mid\;
n \geq \log_k t,\ 1 \le l < k^n
\right\}
\]
is dense in $(0,1)$. Now, this is clear, since  $\mathcal{S}'$ contains all rational numbers in $(0,1)$: For  $p/q$ with $q$ big enough and $0 < p < q$, it suffices to take $n=q$ and
$l = k^p<k^n$. 
\end{proof}

\section{Tameness and nullness of constant length substitution systems}\label{sec:tamenessandnullness}

	Let $\varphi$ be a primitive substitution of  length $k\geq 2$. 
	The amorphic complexity of $X_{\varphi}$ is either zero or lies in $[1,\infty]$.
	One has $\ac(X_{\varphi})=0$ if and only if $X_{\varphi}$ is finite
	and $\ac(X_{\varphi}) = \infty$ if and only if $X_{\varphi}$ does not have discrete spectrum.
	In this section we show that for an infinite $X_{\varphi}$, $\ac(X_{\varphi})=1$ if and only if $X_{\varphi}$ is null and that nullness and tameness of $X_{\varphi}$ are equivalent. 
		To this end, Lemma \ref{prop:kernelandnullness} below will be a crucial step.

\begin{lemma}\label{prop:kernelandnullness}
	Let $k\geq 2$. Let $x$ be a $k$-automatic sequence.
	Assume that  for each $m\geq 1$, $x$ appears at most once as $k^m$-$\AP$ in $x$; that is
	$x_n = x_{i + nk^m}$, $n\in\Z$
	for at most one $0\leq i<k^m$.
	 Then $x$ is null if and only if all  $z\in \Ker_k(x)$, $z\neq x$ are null. 
\end{lemma}

	Before proving Lemma \ref{prop:kernelandnullness}, we consider some  examples which will explain how it will be used in the proof Theorem \ref{introthm:tameness}.
		Recall that for a sequence $x$ and $t\geq 1$, we use the notation 
\[\AP_t(x)=\{(x_{i+tn})_n \mid 0\leq i<t\} \quad\text{and}\quad \Ker_k(x) = \bigcup \{\AP_{k^n}(x)\mid n\in\N\}.
\] 
	The latter set is called the $k$-kernel of $x$; $x$ is $k$-automatic if and only if its $k$-kernel is finite. 
	By a slight abuse of terminology we will say that $(x_{i+tn})_n $ is the $i$th  $t$-\emph{AP} of $x$.

	\begin{example}\label{ex:nullness1}
	Consider the substitution 
	\[
	\varphi(0) = 0012,\quad
	\varphi(1) = 1012,\quad 
	\varphi(2) = 2012
	\]
	of constant length $4$. 
	Note that $\varphi$ is of height 1, has a coincidence, and so, $X_{\varphi}$ has discrete spectrum.
		Let $x =\dots2012.0012\dots$ be a fixed point of $\varphi$.
	The discrepancy substitution of $\varphi$ is a constant function $\varphi_s\colon \A_s\to\A_s$.

	We have $x_n = x_{4n}$, $n\in\Z$; and the $4$-kernel of $x$
	\[\Ker_{4}(x) = \{x, 0^{\omega}, 1^{\omega}, 2^{\omega}\}
	\]
	consists only of $x$ and constant sequences.
	Further, for any $n\geq 1$, $x$ appears only once as $4^n$-$\AP$ in $x$ (namely, as the 0th progression) and all other $4^n$-$\AP$s are constant. 
	Since all constant sequences are  null, by Lemma \ref{prop:kernelandnullness}, the $4$-automatic sequence $x$ is also null.
	\end{example}

Example \ref{ex:nullness1} above suggests that a constant length substitution with only one column set of size $>1$ always gives rise to a null system. Substitution systems with this property have, in fact, already been studied and shown to be null, for instance in the original paper of Goodman \cite[Prop.~6.2]{Goodman}, where the authors restrict to the binary case. Moreover, Huang, Li, Shao, and Ye observed that such a substitution system $X_{\varphi}$ is always an asymptotic extension of its MEF \cite[Lem.~5.1]{HuangLiShaoYe2003}: any $x,y\in X_{\varphi}$ that map to the same point in the MEF under a factor map must be asymptotic.\footnote{Recall that $x$ and $y$ are \emph{asymptotic} if for every $\varepsilon>0$, one has $d(T^n x, T^n y)<\varepsilon$ for all but finitely many $n\in\mathbb{Z}$.}
However, Example \ref{ex:nullness3} below shows that constant length substitution systems can be null while exhibiting a more complicated structure; in particular, they need not be asymptotic extensions of their MEFs.

	 \begin{example}\label{ex:nullness3}
	Consider the substitution 
	\[
	\varphi(0) = 00012, \quad
	\varphi(1) = 12012,  \quad
	\varphi(2) = 20012
	\]
	of constant length $5$. 
	 	Let $x$ be a fixed point of $\varphi$ given as
	 	\[
	 	x=\dots 20012|00012|00012|12012|20012.12012|20012|00012|12012|20012\dots
	 	\]
	 	Note that $\varphi$ is of height 1, $X_{\varphi}$ has discrete spectrum, and the discrepancy substitution of $\varphi$ is given by
	 	\[
	\varphi_s{0\choose 1} = {0\choose 1}{0\choose 2}, \quad 
	\varphi_s{0\choose 2} = {0\choose 2}, \quad
	\varphi_s{1\choose 2} = {1\choose 2}{0\choose 2}.
	\]
	 Here, we have 
	 \[(x_{5n})_{n\in\Z}=x,\quad (x_{2+5n})_{n\in\Z}=0^{\omega},\quad (x_{3+5n})_{n\in\Z}=1^{\omega},\quad (x_{4+5n})_{n\in\Z} = 2^{\omega}.\]
	 Denoting $y=(x_{1+5n})_{n\in\Z}$, we see that the $5$-kernel of $x$ is given by
	 \[\Ker_5{x} =\{x, y, 0^{\omega}, 1^{\omega}, 2^{\omega}\}
	 \]
	  The sequence 
	 \[ y = \dots 00002|00002|00020|00022|00020.20020|00020|00020|20020|00020\dots
	 \]
	 is again $5$-automatic; it has the property that $y_{5n}=y_{n}$, $n\in\Z$, and  all other arithmetic progressions of $y$ of difference $5$ are constant.
	Hence, by Lemma \ref{prop:kernelandnullness}, $y$ is null.
	Now since $y$ is null and constant sequences are null, by another application of Lemma \ref{prop:kernelandnullness}, $x$ is also null.
	One can calculate that the number of nonconstant $5^n$-$\AP$s of $x$ is given by $2n+1$ for $n\in\N$.
	
	To see that $X_{\varphi}$ is not an asymptotic extension of its MEF $\Z_5$, let $x$ and $y$ be the fixed point of $\varphi$ given by
	\begin{align*}
	x &= \dots 2.1\varphi(2012)\varphi^2(2012)\varphi^3(2012)\dots \\
	y &= \dots1.2\varphi(0012)\varphi^2(0012)\varphi^3(0012)\dots
	\end{align*}
	Clearly, $x$ and $y$ have the same image in $\Z_5$ under the factor map; however, $x$ and $y$ are not asymptotic: $\varphi$ is of constant length and $\varphi^n(0)\neq \varphi^n(2)$ for all $n\geq 1$.
	\end{example}
	
	In Example \ref{ex:nullness3} the growth of the number of nonconstant $k^n$-$\AP$s  of $x$ is not constant, but it is still small---it is polynomial.
	This growth is controlled by the maximal growth type of the discrepancy substitution and there is a gap between possible growths of these quantities: they are either polynomial---as in the examples above---or exponential (see Proposition \ref{prop:simplekernelstructure} below).
	Not surprisingly, our sequences will turn out to be null exactly when the first alternative happens.
	This motivates the following definition.
	
	\begin{definition}\label{def:simplekernelstructure}
	For a sequence $x$ and integer $k\geq 2$ we define
	\[\d_m(k,x) = \abs{\{0\leq i< k^m\mid (x_{i+nk^m})_n \text{ is nonconstant}\}},\ m\in\N.
	\] 
	We say that a $k$-automatic sequence $x$ of height $1$ has  \textit{simple} $k$-\emph{kernel structure} if $(\d_m(k,x))_m$ has subpolynomial growth, i.e.\ $\d_m(k,x) = O(m^d)$ for some $d\geq 0$. 
	\end{definition}
	
	\begin{remark}\label{rem:simplekernelstructure}
	Note that $\d_m(k,x)\leq \d_{m+1}(k,x)$ for all $m\in \N$ and thus, for any $n\geq 1$, $x$ has simple $k$-kernel structure if and only if it has simple $k^n$-kernel structure.
	\end{remark}

	Recall that for a substitution $\varphi\colon\A\to\A^*$ of constant length $k$, 
	\[\C(\varphi)=\{\varphi^m_j(\A)\mid 0\leq j<k^m,\ m\in\N\}
	\]
	denotes the union of all column sets of $\varphi^m$,  $m\in \N$ (see Section \ref{sec:substitution} for more details).

	\begin{proposition}\label{prop:simplekernelstructure}
	Let $\varphi\colon\A\to\A^*$ be a primitive substitution of length $k\geq 2$ and height 1. Let $x$ be a fixed point of $\varphi$.
	We have 
	\begin{equation}\label{eq:growthdm}
	\d_m(k,x) = \bigT (m^{d_{s}}\lambda_s^m),
	\end{equation}	
	where $(\lambda_s,d_{s})$ is the discrepancy type of $\varphi$.
	Furthermore, assuming  $x$ is not periodic, the following are equivalent:
	\begin{enumerate}
	\item\label{prop:simplekernelstructure1} $x$ has simple $k$-kernel structure,
	\item\label{prop:simplekernelstructure2} $\lambda_s=1$,	
	\item\label{prop:simplekernelstructure3} for each nonconstant $y\in \Ker_k(x)$ and $m\geq 1$, $y$ appears at most once in $y$ as $k^m$-$\AP$,
	\item\label{prop:simplekernelstructure4} for each $\mathcal{A}'\in \C(\varphi)$ with $\abs{\A'}\geq 2$ and each $m\geq 1$ there exists at most one $0\leq i<k^m$ such that $\varphi^m_i(\mathcal{A}')=\mathcal{A}'$.
\end{enumerate}		
	\end{proposition}
	\begin{proof}

	First we will show  that \eqref{prop:simplekernelstructure1} and \eqref{prop:simplekernelstructure3} are equivalent.
	For ease of notation we will write $\d_m(x)$ instead of $\d_m(k,x)$ (with the convention that $k$ is changed to $k^n$ whenever we change $\varphi$ to its power $\varphi^n$).
	
	Let $\NK_k(x)$ denote the nonconstant elements of  $\Ker_k(x)$ and let  $t=\abs{\NK_k(x)}$; note that $x\in \NK_k(x)$.
	For each $y,z\in \NK_k(x)$, let $a_{yz}$ be the number of times $y$ appears in $z$ as $k$-AP.
	 Let $A=[a_{yz}]$ be a $t\times t$ integer matrix indexed by $\NK_k(x)$.
	 Write $A^m=[a^{(m)}_{yz}]$ for $m\in\N$, and let $\lambda$ be the dominant eigenvalue of $A$.
	Since for any constant sequence, its $k$-APs are also constant, $a^{(m)}_{yz}$ is the number of times $y$ appears in $z$ as $k^m$-AP.
	Furthermore, the number of nonconstant $k^m$-$\AP$s of $x$ is given by the sum of the elements of the column of $A^m$ at index $x$. That is, 
	\begin{equation}\label{eq:dmgrowthanother}\d_m(x) = \sum_{z\in \NK_k(x)} a^{(m)}_{zx}= \bigT(||A^m||_1) =\bigT(\lambda^m m^d),
	\end{equation}
for some $d\in\N$,	where we use Lemma \ref{lem:growth_of_letters}\eqref{lem:growth_of_letters1} for the last equality. 
	Since  the dominant eigenvalue $\lambda$ of an integer matrix $A$ is $>1$ if and only if some of its power $A^m$ has at least one integer $\geq 2$ on the diagonal (i.e.\ when $a^{(m)}_{yy}\geq 2$ for some $y\in \NK_k(x)$), this shows the  equivalence.
			
	To see Formula \eqref{eq:growthdm} and the equivalences between \eqref{prop:simplekernelstructure1}, \eqref{prop:simplekernelstructure2}, and \eqref{prop:simplekernelstructure3}, we will calculate $\d_m(x)$ is another way.
	 By Remark \ref{rem:simplekernelstructure} and \eqref{eq:dmgrowthanother}, we may change $\varphi$ to some of its power. 
	 Hence, we may assume without loss of generality that the discrepancy substitution $\varphi_s$ is in normal primitive form.
	
	For each distinct pair of letters $\alpha=\{a,b\}$, let $(\lambda_{\alpha}, d_{\alpha})$ be its growth type; in particular for each distinct pair of letters $\alpha=\{a,b\}$ and all $m$ big enough we have
	\begin{equation}\label{eq:2growth}
	1/2  c_{\alpha}\lambda_{\alpha}^mm^{d_{\alpha}}\leq \abs{\varphi_s^{m}(\alpha)} \leq 2 c_{\alpha}\lambda_{\alpha}^mm^{d_{\alpha}}
	\end{equation}
	for some $c_{\alpha}>0$.
	Let $\alpha'$ be any pair of letters of maximal growth type $(\lambda_{\alpha'}, d_{\alpha'}) = (\lambda_s, d_{s})$. 
	Note that for each $m\in \N$,  we have
	\[\d_m(x) =\abs{\{0\leq j<k^m\mid \abs{\varphi^m_j(\A)}\geq 2  \}}.
	\]
	Hence, by \eqref{eq:2growth}, for all $m$ big enough,
	\[1/2 c_{\alpha'}\lambda_{\alpha'}^mm^{d_{\alpha'}} \leq \abs{\varphi_s^{m}(\alpha')}  \leq \d_m(x) \leq \sum_{\alpha} \abs{\varphi_s^{m}(\alpha)} \leq \sum_{\alpha} 2c_{\alpha}\lambda_{\alpha}^mm^{d_{\alpha}}.
	\]
	Hence, $\d_m(x) = \bigT (m^{d_{s}}\lambda_s^m)$, which shows Formula \eqref{eq:growthdm}.
			Thus, in Formula \eqref{eq:dmgrowthanother}, we have, in fact, $\lambda = \lambda_s$ and $d=d_{s}$, and  the equivalences  between \eqref{prop:simplekernelstructure1}, \eqref{prop:simplekernelstructure2}, and \eqref{prop:simplekernelstructure3} follow.
	 
	Finally, we  show  that \eqref{prop:simplekernelstructure3} and 
    \eqref{prop:simplekernelstructure4} are equivalent.
        First note that $\A'$ lies in $C(\varphi)$ (i.e.\ $A'$ is a column set of $\varphi^n$ for some $n\geq 1$) if and only if there is $y\in \Ker_k(x)$ which has symbols only from $\A'$; this $y$ is nonconstant if and only if $\abs{A'}\geq 2$. 
	 Furthermore, $y$ appears as $j$th $k^m$-AP of itself if and only if $\varphi^{m}_{j}|_{\A'}\colon \A'\to \A'$ is the identity map. 
     This immediately shows the implication from \eqref{prop:simplekernelstructure4} to \eqref{prop:simplekernelstructure3}. 
     For the converse implication it is enough to note that if \eqref{prop:simplekernelstructure4} does not hold for some $|\A'|\geq 2$, then by taking $\varphi$ to some further power $\varphi^m$ if necessary, we can ensure that, for some distinct $0\leq i',j'<k^{m}$ the maps
\[\varphi^{m}_{i'}|_{\A'}\colon \A'\to \A'\quad \text{and}\quad \varphi^{m}_{j'}|_{\A'}\colon \A'\to \A'
\]	
	 are both identities. 
		\end{proof}

	\begin{remark}\label{rem:graph}
	 In \cite{CovenQuasYassawi}, the authors associate with each primitive substitution $\varphi\colon\A\to\A^*$ of  length $k\geq 2$ a certain directed graph $\mathcal G_\varphi$. 
	If $\varphi$ has height 1, the graph  $\mathcal G_\varphi$ is defined as follows: its vertices are given by $\C(\varphi)$ and there is an oriented edge from $\mathcal{B}$ to $\mathcal{C}$ if $\varphi_j(\mathcal{B})=\mathcal{C}$ for some $0\leq j<k$. If $\varphi$ does not have height 1, then $\mathcal G_\varphi$ is defined as the graph of its pure base.
	It is straightforward to note that Condition \eqref{prop:simplekernelstructure4} in Proposition \ref{prop:simplekernelstructure} corresponds to the fact that    $\mathcal G_{\varphi}$ does not contain two distinct cycles which share a common vertex. 
	In \cite[Thm.\ 2.2]{FuhrmannKellendonkYassawi2024} it is shown that this condition precisely characterises  systems $X_{\varphi}$ that are  tame. 
	\end{remark}

	We will now show Lemma \ref{prop:kernelandnullness}. We will need the following observation.
\begin{lemma}\label{lem:combinatorics}
	Let $k\geq 2$ and $s\geq 0$.
	 There exists $S\in \N$ such that for any set $F\subset \Z$ of size at least $S$ there are $m\geq 1$ and a residue $0\leq t<k^m$  such that:\begin{enumerate}
\item\label{lem:combinatorics1} $\abs{F\cap (t+k^m\Z)}\geq s$,
\item \label{lem:combinatorics2} $\abs{F\cap  (\Z\setminus (t+k^m\Z))}\geq s$.
\end{enumerate} 
\end{lemma}
\begin{proof}

	Let \( S = k s^2 \). Let \( F \subset \mathbf{Z} \) be any set with \( |F| \geq S \). Without loss of generality we may assume that $s\geq k$.

Choose \( m \) to be the smallest integer such that $F$ has nonempty intersection with at least \( s+1 \) residue classes modulo \( k^m \).
	 Note that $m\geq 2$.
	 Since \( m \) is minimal, this implies that $F$ has nonempty intersection with at most \( s \) residue classes modulo \( k^{m-1} \).
	  Hence, there exists a residue  $0\leq t' < k^{m-1}$ such that
\[
\abs{F'} \geq |F|/{s} \geq k s,
\]
where $F' = F\cap (t'+k^{m-1}\Z)$.
	Since the set \( F' \) is partitioned into at most \( k \) different residue classes modulo \( k^m \), there exists $t\in \{t'+ik^{m-1}\mid 0\leq i<k\}$ such that:
\[
\abs{F\cap (t+k^{m}\Z)}  \geq \abs{F'}/k \geq \frac{k s}{k} = s
\]
and \eqref{lem:combinatorics1} holds.
	By our choice of \( m \),  there are at least \( s+1 \) nonempty residue classes modulo \( k^m \) in $F$; in particular, $s$ of them are different than $t$ and so \eqref{lem:combinatorics2} also holds.
\end{proof}

\begin{proof}[Proof of Lemma \ref{prop:kernelandnullness}]

Let $x\in \A^{\Z}$ be $k$-automatic. Clearly, if $x$ is null, then all subsequences of $x$ are null. In particular all sequences in $\Ker_k(x)$ are null.

To prove the other direction, assume that for each $m\geq 1$ there exists at most one $0\leq j<k^m$ with
\begin{equation}\label{eq:degeneretekernel}
x_n=x_{j + k^mn}, \ n\in\Z,
\end{equation}
and all other $z\in \Ker_k(x)$, $z\neq x$ are null. 
	We need to show that $x$ itself is null.

	  Let $Z=\Ker_k(x)\setminus \{x\}$; by assumption $Z$ is a finite set  of null sequences. 
	  Let $s'\in\N$ be such that for all $F\subset \Z$ of size $s'$ we have that
	  \begin{equation}\label{eq:parameter s'}
	  \{a,b\}^{s'}\nsubseteq\bigcup \{\L(z,F)\mid z\in Z\}
	  \quad \text{for all distinct}\quad a,b\in\A.
	  \end{equation} 
	  (It is enough to take $s'$ equal to the sum of nullness parameters of $z\in Z$.)
	  Observe that if for some $m\geq 1$ there is no $0\leq j< k^m$ such that \eqref{eq:degeneretekernel} holds, then the set of $k^m$-$\AP$s of $x$ is contained in $Z$. 
	 In this case, it it easy to see that $x$ is $((k+1)s')$-null: in any $F\subset \Z$ of size at least $(k+1)s'$ one can find a subset of $F'\subset F$ of size $s'$ lying in the same residue class modulo $k$ and use \eqref{eq:parameter s'}.
	 Hence, for the rest of the proof we may assume that for each $m\geq 1$ there is $j$ such that \eqref{eq:degeneretekernel} holds.

	   Put $s = (\abs{Z}+1)s'$
	  and let $S\in\N$ be given by  Lemma \ref{lem:combinatorics} applied to $k$ and $s$. 
	  We will show that $x$ is $S$-null. 
	  For this let $F = \{a_1<\dots <a_S\}\subset \Z$ be any set of size $S$.
	   By Lemma \ref{lem:combinatorics}, there exist $m$ and two sets $F_1,F_2\subset F$   of size $s$ such that:
	  \begin{enumerate}
\item all elements of $F_1$ have the same residue $t$ modulo $k^m$, and
\item all elements of $F_2$ are different than $t$ modulo $k^m$.
\end{enumerate} 
	Since we are ultimately interested in the set 
	\[\L(x,F)=\{w\in\A^S\mid x_{i+a_1}\dots x_{i+a_S} = w \text{ for some } i\in\Z\} = \{T^i(x)|_{F}\mid i\in\Z\},
	\] we may assume without loss of generality that $t=0$. 
	Furthermore, we can also assume that $j=0$ for $m=1$ in  \eqref{eq:degeneretekernel}, and consequently,   
	 $j=0$ for all $m\geq 1$ in  \eqref{eq:degeneretekernel}.
	For $N\subset\Z$ and $H\subset \Z$ we will use notation 
	\[\L(x,H,N) = \{T^i(x)|_{H}\mid i\in N\}.
	\]

 	Let $a,b\in\A$ be two distinct letters. To show that
\[\{a,b\}^S\nsubseteq \L(x,F),
\] 	
 	 it is enough to show that 
 	\begin{enumerate}[(a)]
 	\item\label{nulcon1} $ \{a,b\}^s \nsubseteq \L(x,F_1, \Z\setminus k^m\Z)$, and
 	\item\label{nulcon2} $ \{a,b\}^s \nsubseteq \L(x,F_2, k^m\Z).$
 	\end{enumerate} 
 	To see that this is indeed enough, let $w_1$ (resp.\ $w_2$) be a word in $\{a,b\}^s$ that does not lie in $\L(x,F_1, \Z\setminus k^m\Z)$ (resp.\  $\L(x,F_2, k^m\Z)$). 
 	Suppose that $\{a,b\}^S \subset \L(x,F)$.
	Then there is $n\in \Z$ such that $T^n(x)|_{F_1}=w_1$ and 
	$T^n(x)|_{F_2}=w_2$.
	If $n$ is divisible by $k^m$, then $w_2 \in \L(x,F_2, k^m\Z)$; otherwise $w_1\in \L(x,F_1, \Z\setminus k^m\Z)$. 		In any case we get a contradiction.
 	
 	 To see \ref{nulcon1}, recall that all points in $F_1$ are divisible by  $k^m$, and that all  sequences in $\AP_{k^{m}}(x)$ except the 0th arithmetic progression $(x_{nk^{m}})_n$ lie in $Z=\Ker_k(x)\setminus \{x\}$. Hence,
 	 \[\L(x,F_1, \Z\setminus k^m\Z) = \bigcup \{\L(y,F'_1)\mid y\in Z\},\]
 	 where $F'_1 = 1/k^m F_1$, and the claim follows from \eqref{eq:parameter s'} and the choice of $s\geq s'$.
 	 
 	 To see  \ref{nulcon2}, recall that all elements in $F_2$  are not divisible by $k^m$. Divide the nonzero residues modulo $k^m$ into $\abs{Z}$ disjoint classes:
 	 \[W_z = \{ 1\leq i<k^m\mid (x_{i+nk^m})_n = z\},\quad z\in Z.
 	 \]
 	 Since $F_2$ is of size $(|Z|+1)s'$ we can find $F_3\subset F_2$ of size $s'$ and $z\in Z$ such that 
	\[\abs{F_3\cap (W_z + \Z k^m)}\geq s'.
	\] 	 
 	 Now, if $\abs{i - j}<k^m$ for some distinct $i,j\in F_3$, then  $x_{n+i} = x_{n+j}$ for all $n\in k^m\Z$ (since they correspond to the same letter of the sequence $z$) and the claim is true since we cannot realize two different letters on positions $n+i$ and $n+j$ for any $n\in k^m\Z$.
 	 
 	 In the other case, consider the map $i\to i'$, which maps each $i\in F_3$ to the biggest $i'\in \Z$ such that $i'k^m<i$.
 	 Since all $i\in F_3$ are at least $k^m$ apart, this map is injective and
 	 \[ \L(x,F_3,k^m\Z)  = \L(z,F'_3),\]
where $F'_3$ denotes the set of all $i'$, $i\in F_3$.  	 
 	Since $\abs{F'_3}=s'$ and $z$ is not $s'$-null, the claim follows.
\end{proof}

Now we are ready to prove the main result of this section. 
	The main part will be to show that any system $X_{\varphi}$ with amorphic complexity 1 is null.
	Let $x$ be a fixed point of $\varphi$.
	We will see that if $\ac(X_{\varphi})=1$, then each nonconstant element $y$ of $\Ker_k(x)$ always eventually leads---in  the  manner suggested by Example \ref{ex:nullness3}---to sequences whose all but one $k$-$\AP$ are constant. 
	Since all constant sequences are null, a repeated use of Lemma \ref{prop:kernelandnullness} will let us conclude that $y$ is null. In particular, $x$ is null.

\begin{theorem}\label{thm:tameness} 
Let $\varphi\colon \A\to\A^*$ be a primitive substitution of length $k\geq 2$  and assume that $X_{\varphi}$ is infinite. The following are equivalent:
\begin{enumerate}
\item\label{thm:tameness1}  $\ac(X_{\varphi})=1$.
\item\label{thm:tameness2}  $X_{\varphi}$ is null,
\item\label{thm:tameness3}  $X_{\varphi}$ is tame.
\end{enumerate}
\end{theorem}
\begin{proof} 
By claims \eqref{thm:dekking1}, \eqref{thm:dekking4}, \eqref{thm:dekking5} of Theorem \ref{thm:dekking}, we may assume without loss of generality that $\varphi$ has height $1$.
	  Since $X_{\varphi} = X_{\varphi^m}$ for any $m\geq 1$, we may furthermore assume that $\varphi$ has some fixed point $x\in X_{\varphi}$.
	 By Theorem \ref{thm:ac_of_automatic_systems}, $\ac(X_{\varphi})=1$ if and only if  $\lambda_s=1$, where $\lambda_s$ is the discrepancy rate of $\varphi$.
	
	The fact that \eqref{thm:tameness2} implies \eqref{thm:tameness3} is well known (and follows e.g.\ directly from Definition \ref{def:general tame null}).
	The fact that \eqref{thm:tameness3} implies \eqref{thm:tameness1} follows directly from Proposition \ref{prop:simplekernelstructure} and \cite{FuhrmannKellendonkYassawi2024} (see also Remark \ref{rem:graph}).	
	Thus we only need to show that \eqref{thm:tameness1} implies \eqref{thm:tameness2}. 
	By Lemma \ref{lem:nullness}, it is enough to show that the fixed point $x$ is null.

	For each $y$ in the $k$-kernel of $x$, we let $\NK(y)$ denote the set of nonconstant sequences lying in the $k$-kernel of $y$; note that this set always includes $y$ itself unless $y$ is constant.
    Since $\lambda_s=1$, by Proposition \ref{prop:simplekernelstructure}, $x$ has a simple $k$-kernel structure and	
		\begin{enumerate}
	\item [(C1)] for each $y\in \NK(x)$ and each $n\geq 1$, $y$ appears at most once as $k^n$-AP in $y$.
	\end{enumerate}
	Let $n'$ be the least common multiple of all the smallest $n$'s such that $y$ appears as $k^n$-$\AP$ in $y$ if such $n$ exists.
	Changing $\varphi$ to its power $\varphi^{n'}$, 
 	we may furthermore assume that
   \begin{enumerate}
   \item [(C1')] for each $y\in \NK(x)$ either $y$ does not appear as $k^n$-$\AP$ of $y$ for any $n\geq 1$ or $y$ appears exactly once as $k^n$-$\AP$ in $y$ for all $n\geq 1$.
   \end{enumerate}
	By Lemma \ref{prop:kernelandnullness}, claim (C1), and the fact that all constant sequences are null we have that
   \begin{enumerate}
   \item [(C2)]  $y\in\NK(x)$ is null if and only if all sequences in $\NK(y)\setminus\{y\}$ are null.
   \end{enumerate}

Now, we will show that $x$ is null. Suppose for the sake of contradiction that it is not and let $y_0=x$.
By (C2), there exists a nonnull sequence $y_1\in \NK(y_0)\setminus \{y_0\}$. Using (C2) again, we conclude that there exists a nonnull sequence $y_2\in \NK(y_1)\setminus \{y_1\}$. Continuing in this way, by (C2), we construct an infinite sequence of nonnull sequences $y_0, y_1, y_2, \dots$ lying in $\Ker_k(x)$ such that
\[y_{i+1}\in \NK(y_i)\setminus \{y_i\} \quad \text{for} \quad i\geq 0.
\]
    By (C1'), all sequences $y_i$, $i\geq 0$ are distinct; indeed, if $y_i=y_j$ for some $j<i$, then $i>j+1$ by construction and $y_i=y_j$ appears as a $k^{i-j-1}$-$\AP$ of some $k$-$\AP$ $z\neq y_i$ of $y_i$. Thus, by (C1'), $y_i$ also appears as a $k$-$\AP$ of $y_i$ and thus, $y_i$ appears at least twice as a $k^{i-j}$-$\AP$ of $y_i$, which contradicts (C1'). Hence, all $y_i$ are distinct. Now, this is clearly a contradiction: all sequences $y_i$ lie in $\Ker_k(x)$ and  $\Ker_k(x)$ is finite.
\end{proof}

\bibliographystyle{alpha}
\bibliography{lit}

\end{document}